\documentclass[a4paper,11pt]{amsart}
\usepackage{graphicx} 
\usepackage{amsmath}
\usepackage{amsfonts}
\usepackage{amsthm}
\usepackage{amssymb}
\usepackage{wrapfig}
\usepackage[shortlabels]{enumitem}
\usepackage{tikz}
\usepackage{esint}
\usepackage{mathrsfs}
\usepackage{physics}
\usepackage{fancyhdr}
\usepackage{microtype}
\usepackage[normalem]{ulem} 
\usepackage{xparse}
\usepackage{framed}

\usepackage{xcolor,ulem,mdframed}

\usepackage[left=3.25cm,right=3.25cm,top=3cm,bottom=2.5cm]{geometry}
\usepackage[colorlinks=true, linktocpage=true, linkcolor=red!70!black, citecolor=green!50!black]{hyperref}

\usepackage{enumitem}

\numberwithin{equation}{section}

\renewenvironment{shaded}{%
	\color{gray!60}%
	\MakeFramed{\FrameRestore}%
}{%
	\endMakeFramed%
}

\title[$\mathbb{A}$-quasiconvexity, partial regularity and Orlicz growth]{Partial regularity for $\mathbb{A}$-quasiconvex \\  functionals with Orlicz growth}

\newcommand{\R}{\mathbb{R}}

\newcommand{\locc}{\operatorname{loc}}
\newcommand{\lebe}{\operatorname{L}}
\newcommand{\dif}{\operatorname{d}\!}
\newcommand{\sobo}{\operatorname{W}}
\newcommand{\hold}{\operatorname{C}}

\newtheorem{defi}{Definition}[section]
\newtheorem{lemma}[defi]{Lemma}

\newtheorem{bem}[defi]{Remark}
\newtheorem{satz}[defi]{Theorem}
\newtheorem{beisp}[defi]{Example}


\begin{document}
	
	\author{Paul Stephan}
	\address{P.S.: Department of Mathematics and Statistics, University of Konstanz, Universit\"{a}tsstra\ss e 10, 78464 Konstanz, Germany}
	\email{paul.stephan@uni-konstanz.de}
	
	\maketitle

	\begin{abstract}
		We establish partial regularity results for minimizers of a class of functionals depending on differential expressions based on elliptic operators. Specifically, we focus on functionals of Orlicz growth with a natural strong quasiconvexity property. In doing so, we consider both $\Delta_{2}\cap\nabla_{2}$-Orlicz growth scenarios and, as a limiting case, $L \log L$-growth. Inspired by {\textsc{Conti \&
				Gmeineder} (Calc. Var. \& PDE 61:215, 2022)}, the proofs of our main results are accomplished by reduction to the case of full gradient partial regularity results.
	\end{abstract}
	
	\tableofcontents
	
	\section{Introduction}
	Let $\Omega\subset\R^{n}$ be open and bounded with Lipschitz boundary. Numerous variational problems, such as those in (stationary) elasticity or fluid mechanics, model key quantities such as displacements or velocities using functionals of the type 
	\begin{equation}
		\widetilde{\mathcal{F}}[u] :=  \int_\Omega F(\varepsilon(u)) \dd{x}. \label{eq:intr}
	\end{equation}
	Here, $\varepsilon(u) = \frac{1}{2} \left( \nabla u + (\nabla u)^\top \right)$ denotes the symmetric gradient of a map  $u\colon\Omega\to\R^{n}$. Different models, however,  involve more general operators than the symmetric gradient.  
	In view of a general framework which comprises \eqref{eq:intr} as a special case, we replace $\varepsilon$ by an arbitrary homogeneous differential operator $\mathbb{A}$, meaning that $\mathbb{A}$ has a  representation \begin{equation}\label{eq:diffop}\mathbb{A}u := \sum_{i=1}^n A_i \frac{\partial}{\partial x_i} u,\qquad u\colon \Omega\subset\R^{n}\to \R^{k}, 
	\end{equation}
	for certain linear maps {$A_{i}\colon \mathbb{R}^k \to \mathbb{R}^l$.} In this general framework, we consider functionals of the form
	\begin{equation}\label{eq:functionalform}
		\mathcal{F}[u] := \int_\Omega F(\mathbb{A}u) \, \dd{x}.
	\end{equation}
	By making different choices of $F$ and $\mathbb{A}$, functionals of the form \eqref{eq:functionalform} provide a versatile approach to various energy minimization problems not only in elasticity, but also in broader contexts, such as continuum mechanics and fluid dynamics, {and also in problems from general relativity}. For a detailed discussion on these applications, see, for instance, \cite{bartnik_constraint_2004,BehnGmeinederSchiffer, conti_symmetric_2020, dain, FonsecaMuller, friedrichs_boundary-value_1947, fuchs_variational_2000}. Once  {appropriate} boundary conditions and semiconvexity assumptions on the integrands are imposed, it is then possible to establish the existence of minimizers of functionals \eqref{eq:functionalform} in suitable spaces of weakly differentiable functions.
	
	Once this is achieved, it is then natural to inquire as to which conditions have to be assumed on $F$ for minimizers to share better regularity properties than generic competitors. In the present paper, we deal with this question in the realm of gradient H\"{o}lder regularity. Since the minimization of \eqref{eq:functionalform} is a genuinely vectorial problem, it is well-known that minimizers are not necessarily fully $\hold^{1, \alpha}$-regular in general (see, e.g., \cite{DeGiorgi1968, giusti1968esempio, mazya_examples_1968}). Consequently, a \textsc{De Giorgi-Nash-Moser} theory is not available. In this regard, a suitable substitute is provided by the ($\hold_{\locc}^{1,\alpha}$-)\emph{partial regularity} of  minimizers, meaning that they are of class $\hold^{1,\alpha}_\text{loc}(U)$ for all $0<\alpha<1$, where $U\subset\Omega$ is a relatively open set with full Lebesgue measure: $\mathscr{L}^{n}(\Omega\setminus U)=0$.
	
	In view of the direct method of the Calculus of Variations and based on the fact that the minimization of \eqref{eq:functionalform} is a vectorial problem, it is natural to suppose that $F$ satisfies a strong variant of \textsc{Morrey}'s \emph{quasiconvexity}, implying both lower semicontinuity and coercivity, for example, on Dirichlet subclasses of Sobolev spaces.
	If $\mathbb{A}=\nabla$ is the usual gradient, a rather complete picture of partial regularity is available in this situation; see \cite{acerbi_regularity_2001, bärlin2022mathcalaharmonic, carozza_partial_1998, carozza_regularity_1996,
		Filippis1,  Filippis2, diening_partial_2012, evans_quasiconvexity_1986,  gmeineder_partial_2019-1, gmeineder_quasiconvex_2024,  kristensen_partial_2003, kuusi_partial_2016, DarkSide,schmidt1, schmidt_regularity_2009} for a non-exhaustive list. More precisely, starting with \textsc{Evans} \cite{evans_quasiconvexity_1986} for strongly quasiconvex functionals with $p$-growth for $2 \leq p < \infty$, \cite{carozza_partial_1998, carozza_regularity_1996} for $1 < p \leq 2$, and finally \cite{gmeineder_partial_2019-1} for $p=1$, partial regularity results are now available for all $p$-growth functionals. Here, we say that $F\in\hold(\R^{N\times n})$ has $p$-growth, $1\leq p<\infty$, if there exists a constant $c>0$ such that 
	\begin{align}\label{eq:pgrowth}
		|F(z)|\leq c\,(1+|z|^{p})\qquad\text{for all}\;z\in\R^{N\times n}. 
	\end{align}
	{In \cite{schmidt1,schmidt_regularity_2009}, \textsc{Schmidt} generalized these results to functionals of $(p,q)$ growth for $p>1$.}  The works of \textsc{Diening} et al. \cite{diening_partial_2012} and \textsc{Breit \& Verde} \cite{breit_quasiconvex} extended {this framework} to the context of Orlicz growth, in turn using and modifying \textsc{Duzaar \& Steffen's} method of $\mathcal{A}$-harmonic approximation; see \cite{duzaar_optimal_2003}. {Finally, \textsc{Gmeineder} \& \textsc{Kristensen} provided the generalization for the case $p=1$ in \cite{gmeineder_quasiconvex_2024}.}
	
	While the aforementioned  results deal with the case $\mathbb{A}=\nabla$, there are further results on the interaction of partial regularity and general operators $\mathbb{A}$. Based on earlier work by \textsc{Gmeineder} \cite{gmeineder_partial_2021} for the symmetric gradient, \textsc{Conti \& Gmeineder} established in \cite{contiFXG} that, under a suitable $\mathbb{A}$-quasiconvexity condition on  $F$, the partial regularity of minimizers of $p$-growth functionals \eqref{eq:functionalform} with $1<p<\infty$ is equivalent to $\mathbb{A}$ being \emph{elliptic}. This means that the Fourier symbol map $$\mathbb{A}[\xi] := \sum_{i=1}^n A_i \xi_i \colon \R^k \to \R^l$$ is injective for all phase space variables $\xi \in \R^n \backslash \{0\}$. 
	
	Whereas the $p$-growth assumption is satisfied in numerous situations, there are plenty of models that require working with integrands that satisfy a more flexible condition than \eqref{eq:pgrowth}. An instance of this include models that involve logarithmic hardening of materials, or stationary models for fluids of \textsc{Prandtl-Eyring}-type; see \cite{BREIT20121910,FrehseSeregin,Fuchs1999VariationalMF,fuchs_variational_2000} for more detail.  {Note that the energies underlying these references are usually convex. However, as shall be discussed in detail below, the reduction to the partial regularity for \emph{full gradient} functionals can only be accomplished by viewing them as suitable $\mathbb{A}$-quasiconvex functionals.}

	To provide a unified approach that includes such models, the present paper deals with strongly $\mathbb{A}$-quasiconvex functionals of \emph{Orlicz growth}. In this context, our first result is as follows: 
	\begin{satz}[Partial regularity, $\Delta_2 \cap \nabla_2$-growth] \label{OrliczMainThm}
		Let $\Omega \subset \mathbb{R}^n$ be open and bounded. Moreover, let $\varphi \in \Delta_2 \cap \nabla_2$, and suppose that $\mathbb{A}$ is a first order constant rank differential operator of the form \eqref{eq:formA}, and that the variational integrand $F: {\mathbb{R}^l} \to \mathbb{R}$ satisfies the following conditions:
		\begin{enumerate}[label=(F\arabic*)]
			\item $F \in \mathrm{C}^{2}({\mathbb{R}^l})$. \label{F1}
			\item There exists a constant $c > 0$ such that $$|F(z)| \leq c(1+\varphi(|z|))$$ holds for all $z \in {\mathbb{R}^l}$. \label{F2}
			\item\label{item:F3} $\varphi$-strong $\mathbb{A}$-quasiconvexity: For every $M > 0$, there exists $\nu_M > 0$ such that
			\[
			\int_{(0,1)^n} \bigl( F(z_0 + \mathbb{A} \zeta) - F(z_0) \bigr) \, \dd{x} \geq \nu_M \int_{(0,1)^n} \varphi_{1+|z_0|}(|\mathbb{A} \zeta|) \, \dd{x}
			\]
			holds for all $\zeta \in \mathrm{C}^{\infty}_c((0,1)^n;{\mathbb{R}^k})$ and $|z_0| \leq M$. Here, $\varphi_{1+|z_0|}$ is the \emph{shifted $N$-function} as defined in \eqref{eq:shift} below. \label{F3}
		\end{enumerate}
		Then the following statements are equivalent:
		\begin{enumerate}[(a)]
			\item $\mathbb{A}$ is elliptic.
			\item Every local minimizer $u \in \lebe^\varphi_\mathrm{loc}(\Omega;  {\mathbb{R}^k})$ with $\mathbb{A}u \in \lebe^\varphi_\mathrm{loc}(\Omega;  {\mathbb{R}^l})$ of $\mathcal{F}$ as in \eqref{eq:functionalform}  is $\mathrm{C}^{1,\alpha}_{\mathrm{loc}}$-partially regular.
		\end{enumerate}
	\end{satz}
	
	In stating the previous theorem, we made use of various notions to be briefly addressed now: Firstly, the condition $\varphi\in\Delta_{2}\cap\nabla_{2}$ on the $N$-function $\varphi\colon [0,\infty)\to[0,\infty)$ means that $\varphi$ neither grows too slowly nor too fast; see Section \ref{sec:prelims} for the precise definition. Secondly, \ref{item:F3} is a version of \textsc{Fonseca \& M\"{u}ller}'s $\mathscr{A}$-quasiconvexity \cite{FonsecaMuller} adapted to the setting considered here; see \textsc{Raita} \cite{Raita} for the connections between $\mathscr{A}$- and $\mathbb{A}$-quasiconvexity. In the form as given here, \ref{item:F3} not only implies lower semicontinuity in Orlicz-Sobolev spaces, but also the coercivity of the functionals \eqref{eq:functionalform} on Dirichlet classes; see Section \ref{sec:quasiconvex} for more detail. Lastly, the notion of \emph{local minimality} underlying \ref{item:F3} is classical: We say that a function $u\in\sobo_{\locc}^{1,\varphi}(\Omega; {\mathbb{R}^k})$ is a \emph{local minimizer} provided that 
	\begin{equation*}
		\int_\omega F(\mathbb{A}u) \, \dd{x} \leq \int_\omega F(\mathbb{A}(u+\zeta)) \, \dd{x}
	\end{equation*}
	holds for every open ${\omega} \Subset \Omega$ and all maps $\zeta \in \sobo_{0}^{1,\varphi}(\omega; {\mathbb{R}^k})$.
	
	Following \textsc{Conti \& Gmeineder} \cite{contiFXG}, our strategy to establish Theorem \ref{OrliczMainThm} is \emph{not} to execute an entire partial regularity proof by employing, for example, the methods of blow-up, $\mathcal{A}$-harmonic approximation or similar techniques. Instead, we aim to reduce Theorem \ref{OrliczMainThm} to the main results of \cite{diening_partial_2012,Irving2021partial} by using suitable Korn-type inequalities adapted to the differential operator $\mathbb{A}$. Inequalities of this sort have recently attracted significant interest, as seen in  \cite{breit_trace-free_2017, breit_sharp_2012}, and are of a particularly simple form if $\varphi\in\Delta_{2}\cap\nabla_{2}$. Roughly speaking, if $\mathbb{A}u\in\lebe^{\varphi}(\Omega;{\mathbb{R}^l})$ for $\varphi\in\Delta_{2}\cap\nabla_{2}$, then $\nabla u\in\lebe^{\varphi}(\Omega;{\mathbb{R}^k}\times\R^{n})$, and it is this insight that allows us to reduce Theorem \ref{OrliczMainThm} to the findings of \cite{diening_partial_2012,Irving2021partial} in Section \ref{Proof}. 
	
	Variational integrands $F$ of $\varphi$-growth, $\varphi\in\Delta_{2}\cap\nabla_{2}$, do not always provide the right framework; examples include models with logarithmic hardening. Here, $\varphi(t)=t\log(1+t)$, whereby  $\varphi$ is of class $\Delta_{2}\setminus\nabla_{2}$. In the special cases where $\mathbb{A}=\varepsilon$ or $\mathbb{A}=\varepsilon^{D}$ with the \emph{deviatoric} symmetric gradient $\varepsilon^{D}u:=\varepsilon(u)-\frac{1}{n}\mathrm{div}(u)E_{n}$, it is known from \textsc{Cianchi} et al. \cite{cianchi_korn_2014,breit_trace-free_2017} that $\mathbb{A}u\in\lebe^{\varphi}(\Omega;\R^{n\times n})$ only implies that $\nabla u\in\lebe^{1}(\Omega;\R^{n\times n})$. As we show in Section \ref{sec:LlogL}, this loss of one logarithm in passing from $\mathbb{A}u$ to $\nabla u$ persists for general elliptic operators. Yet, resorting to the recent work \cite{gmeineder_quasiconvex_2024} of \textsc{Gmeineder \& Kristensen}, this still allows us to implement a variant of the above reduction scheme. More precisely, the second main result of the present paper is as follows: 
	\begin{satz}[Partial regularity -- $L \log L$-growth] \label{OrliczMainThm2}
		Let $\Omega \subset \R^n$ be open and bounded, and let $\varphi: \R_{\geq 0} \to \R_{\geq 0}$ be given by $\varphi\colon t \mapsto t \log(1+t)$. Moreover, suppose that the variational integrand $F:{\mathbb{R}^l} \to \mathbb{R}$ satisfies \ref{F1}-\ref{F3} from above. If $\mathbb{A}$ is a constant rank operator of the form \eqref{eq:formA}, then every local minimizer $u\in\lebe_{\locc}^{\varphi}(\Omega;{\mathbb{R}^k})$ of $\mathcal{F}$ as in \eqref{eq:functionalform} with $\mathbb{A}u\in\lebe_{\locc}^{\varphi}(\Omega;{\mathbb{R}^l})$ is $\hold_\text{loc}^{1, \alpha}$-partially regular if and only if $\mathbb{A}$ is elliptic. 
	\end{satz}
	Theorem \ref{OrliczMainThm2} can be understood as a limiting case for the reduction strategy from \cite{contiFXG}. Namely, if we pass to $N$-functions $\varphi$ of sub-logarithmic growth, then $\mathbb{A}u\in\lebe^{\varphi}(\Omega;{\mathbb{R}^l})$ does not necessarily imply that $\nabla u\in\lebe_{\locc}^{1}(\Omega; {\mathbb{R}^k}\times\R^{n})$. In this sense, the previous theorem extends the reduction strategy optimally in this direction. We wish to point out that for the end point $p=1$, it is not possible to employ such a reduction strategy for elliptic operators. Here, partial regularity can still be established by use of a variant of $\mathcal{A}$-harmonic approximation for $\mathbb{C}$-elliptic operators; see \cite{bärlin2022partialregularitymathbbaquasiconvexfunctionals}. 
	
	For completeness, we note that it is equally interesting to consider the other end of the scale, for example, integrands of exponential growth as in \textsc{Beck \& Mingione} \cite{beck_lipschitz_2020}. To the best of our knowledge, very little is known about the partial regularity of minimizers even in the full gradient case. Since we are here primarily interested in implementing the reduction strategy of \cite{contiFXG} to functionals \eqref{eq:functionalform}, we thus defer this task to future work. 
	\subsection*{Structure of the paper} We briefly comment on the organization of the paper. In Section \ref{sec:prelims}, we collect background results on Orlicz spaces and differential operators. Section \ref{sec:quasiconvex} is concerned with convenient reformulation of the $\varphi$-strong quasiconvexity assumption, and also displays its connections to coercivity. Section \ref{Proof} then provides the proof of Theorem \ref{OrliczMainThm} in the case of $\Delta_{2}\cap\nabla_{2}$-growth, and Section \ref{sec:LlogL} establishes Theorem \ref{OrliczMainThm2} in the limiting case of integrands with $L\log L$-growth. 
	\section{Preliminaries}\label{sec:prelims}
	\subsection{General notation}
	
	{To enable a broad application of the operator $\mathbb{A}$, we consider our results in the context of finite-dimensional real inner product spaces $V$ and $W$. This generalizes the specific choice of Euclidean spaces $\mathbb{R}^k$ and $\mathbb{R}^l$ from the introduction. Consequently, we consider from now on $u \colon \Omega \to V$ and $\mathbb{A}u \colon  \Omega \to W$. }
	
	We write $ a \lesssim b $ if $ a \leq c b $ and $ a \gtrsim b $ if $ a \geq c b $  for some constant $c > 0 $. If $a \lesssim b$ and $b \lesssim a$, then we write $a \simeq b$. If  necessary, we specify the dependence of the constant $c$. Moreover, we use $\langle \cdot, \cdot \rangle$ to represent the scalar product on $\mathbb{R}^n$.
	\subsection{Orlicz spaces}
	Next, we provide the requisite background material on Orlicz spaces (see \cite{adams_sobolev, ren_theory_1991}). A function $\psi: \mathbb{R}_{\geq 0} \to \mathbb{R}_{\geq 0}$ is called a \emph{Young function} if it is convex, continuous, and satisfies $\psi(0) = 0$ and $\psi(t) \to \infty$ as $t \to \infty$.
	A function $\psi: \mathbb{R}_{\geq 0} \to \mathbb{R}_{\geq 0}$ is called an \emph{$N$-function} if it is differentiable and satisfies the following criteria:
	\begin{itemize}
		\item $\psi(0) = 0$,
		\item $\psi'$ is right continuous and non-decreasing,
		\item $\psi'(0) = 0$,
		\item $\psi'(t) > 0$ for $t > 0$, and
		\item $\lim \limits_{t \to \infty} \psi'(t) = \infty$.
	\end{itemize}
	\noindent
	We say that an \emph{$N$-function} $\psi$ belongs to the class $\Delta_2$ and write $\psi \in \Delta_2$, if there exists a constant $K > 0$ such that $\psi(2t) \leq K \psi(t)$ holds for all $t \geq 0$. In this case, we denote the infimum of such constants $K$ by
	\[
	\Delta_2(\psi) := \inf \left\{ K > 0 : \psi(2t) \leq K \psi(t) \text{ for all } t > 0 \right\}.
	\]
	Analogously to the $\Delta_2$-condition, we define the $\nabla_2$-condition using the {convex} conjugate. Given an $N$ function $\psi$, its {convex conjugate} $\psi^*$ is defined as
	\begin{align}\label{eq:Fenchel}
		\psi^*(t) := \sup_{s \geq 0} (st - \psi(s)), \quad t \geq 0.
	\end{align}
	An $N$-function $\psi$ is of class $\nabla_2$ if its conjugate $\psi^*$ satisfies the $\Delta_2$-condition. We denote this by $\psi \in \nabla_2$ and define
	\[
	\nabla_2(\psi) := \Delta_2(\psi^*).
	\]
	Given an $N$-function $\psi: \R_{\geq 0} \rightarrow \R_{\geq 0}$, we define the \emph{Luxemburg norm} of a measurable function $f: \Omega \to V$ on a measurable set $\Omega$ by
	
	\begin{equation*}
		\|f\|_{\lebe^\psi(\Omega; V)} = \inf \left\{ \lambda > 0 : \int_{\Omega} \psi \left(\frac{|f(x)|}{\lambda}\right) \,\dd{x} \leq 1 \right\}.
	\end{equation*}
	
	\noindent
	The Orlicz space $\lebe^\psi(\Omega;V)$ consists of all measurable functions $f: \Omega \to V$ for which the Luxemburg norm is finite. For instance, if $\psi(t) = |t|^p$, the Orlicz space $\lebe^\psi(\Omega;V)$ coincides with the classical Lebesgue space $\lebe^p(\Omega;V)$. 
	We now define the Orlicz-Sobolev space $\sobo^{k , \psi}(\Omega; V)$ for $k \in \mathbb{N}$ as
	\begin{equation*}
		\sobo^{k, \psi}(\Omega; V) = \left\{ f \in \lebe^\psi(\Omega; V) : \mathrm{D}^\alpha f \in \lebe^\psi(\Omega; V) \text{ for all  } \alpha \text{ with } |\alpha| \leq k \right\},
	\end{equation*}
	
	\noindent
	where $\mathrm{D}^\alpha f$ denotes the weak derivative of $f$ corresponding to the multi-index $\alpha$. The norm in $\sobo^{k, \psi}(\Omega; V)$ is given by
	\begin{equation*}
		\|f\|_{\sobo^{k, \psi}(\Omega; V)} = \sum_{|\alpha| \leq k} \|\mathrm{D}^\alpha f\|_{\lebe^\psi(\Omega; V)}.
	\end{equation*}
	We define \emph{shifted $ N $-functions}, as introduced in \cite{DieningEttwein+2008+523+556, diening_relaxed_2020_1}. Consider an $ N $-function $\varphi \in \hold^1([0, \infty)) \cap \hold^2((0, \infty))$. For $ a \geq 0 $, the shifted $ N $-function $\varphi_a: [0, \infty) \to [0, \infty)$ is defined by:
	\begin{equation}
		\varphi_a(t) := \int_0^t \frac{\varphi'(a+s)}{a+s}s \,\dd{s}, \quad t \geq 0. \label{eq:shift}
	\end{equation}
	A property of the shifted $N$-function is that $\Delta_2(\varphi_a)$ and $\nabla_2(\varphi_a)$ are uniformly comparable for all $a$. More precisely, the following lemma holds.
	
	\begin{lemma}[{{cf. \cite[Remark 38]{diening_relaxed_2020_1}}}] \label{Delta:Nabla:Comp}
		For all $a \geq 0$, we have
		\[
		\Delta_2(\varphi_a) \simeq \Delta_2(\varphi)
		\]
		and
		\[
		\nabla_2(\varphi_a) \simeq \nabla_2(\varphi),
		\]
		where the underlying constants do not depend on $a$.
	\end{lemma}
	
	\subsection{Homogeneous differential operators}
	Our objective is to establish Theorem \ref{thm:IrvMainReg} for a broader class of homogeneous differential operators $\mathbb{A}$ defined by
	\begin{equation}
		\mathbb{A}u := \sum_{i=1}^n A_i \frac{\partial}{\partial x_i} u, \label{eq:formA}
	\end{equation}
	where $A_i: V \to W$ are linear mappings. We define the \emph{Fourier symbol} of $\mathbb{A}$ as
	\begin{align}\label{eq:Fouriersymbol1}
		\mathbb{A}[\xi] := \sum_{i=1}^n A_i \xi_i,\quad  \xi = (\xi_1, \ldots, \xi_n) \in \mathbb{R}^n,
	\end{align}
	and it is moreover customary to put 
	\begin{align*}
		a\otimes_{\mathbb{A}}b:=\mathbb{A}[b]a,\qquad a\in V,\;b\in\R^{n}. 
	\end{align*}
	We say that $\mathbb{A}$ has \emph{constant rank} if the rank of $\mathbb{A}[\xi]$ remains unchanged for all $\xi \in \mathbb{R}^n \backslash \{0\}$. If \[
	\ker(\mathbb{A}[\xi]) = \{0\}
	\]
	holds for all $\xi \in \R^n \backslash \{0\}$, we refer to $\mathbb{A}$ as an \emph{elliptic operator}. Examples of such operators include the gradient $$\nabla u = \left( \frac{\partial u}{\partial x_1}, \dots, \frac{\partial u}{\partial x_n} \right), \quad u: \R^n \to V,$$ and the symmetric gradient $$\varepsilon(u) = \frac{1}{2} \left( \nabla u + (\nabla u)^T \right),\quad u:\R^n \to \R^n.$$
	Later on, we aim to demonstrate the necessity of ellipticity to ensure that local minimizers are $\hold^{1,\alpha}$-regular. To establish this, we will then use the following result:
	\begin{lemma}[cf. {\cite[Cor. 4.3]{contiFXG}}] \label{Kernel}
		Let $\Omega \subset \mathbb{R}^n$ be open, and let $\mathbb{A}$ be an operator of constant rank that is not elliptic. Then there exists $v \in \ker{\mathbb{A}} \cap \lebe^1_{\mathrm{loc}}(\Omega;V)$ such that $v \notin \hold(\omega)$ for any open $\omega \subset \Omega$.
	\end{lemma}
	For our future purposes, it is moreover convenient with the minimal choice of $W$:
	\begin{bem}[Choices of $W$]\label{rem:choiceofW}
		If $\mathbb{A}$ is a differential operator of the form \eqref{eq:formA} with $V\cong \R^{N}$, then it is no loss of generality to assume that $W\hookrightarrow \R^{N\times n}$. In fact, as established in \cite[Section 3]{contiFXG}, we may suppose that $W=\mathscr{R}(\mathbb{A})$ with the essential range 
		\begin{align}\label{eq:RAdef}
			\mathscr{R}(\mathbb{A}) := \mathrm{lin}(\{\mathbb{A}[e_{i}]v_{j}\colon\;i\in\{1,...,n\},\;j\in\{1,...,N\}\})
		\end{align}
		as introduced in \cite{BreitDieningGmeineder}. The space $\mathscr{R}(\mathbb{A})$ as given in \eqref{eq:RAdef} is the smallest space in which all values $\mathbb{A}v(x)$ for $v\in\hold^{1}(\R^{n};V)$ and $x\in\R^{n}$ are contained. The necessity of considering $\mathscr{R}(\mathbb{A})$ is due to the fact that, whenever $\mathbb{A}$ is a $W$-valued differential operator, viewed as an $Z$-valued differential operator for every linear space with $W\hookrightarrow Z$. In this sense, \eqref{eq:RAdef} is the minimal space $W$ with this property. As an explicit example, we note that the symmetric gradient can be viewed as a differential operator of the form \eqref{eq:formA} for both $W=\R^{n\times n}$ and $W=\mathscr{R}(\mathbb{A})=\mathbb{R}_{\mathrm{sym}}^{n\times n}$, but only $\mathbb{R}_{\mathrm{sym}}^{n\times n}$ has the requisite minimality property. 
	\end{bem}

	Elliptic operators come with convenient properties because they provide control over the entire gradient. This is specified by the following Korn-type  inequality. 
	\begin{satz}[Korn-type inequality, $(\Delta_{2}\cap\nabla_{2})$-growth {{\cite[Prop. 4.1]{contiFXG}}}] \label{thm:korn_uni} 
		Let $\mathbb{A}$ be an elliptic differential operator of the form \eqref{eq:formA}. Moreover, let $\psi\in\Delta_{2}\cap\nabla_{2}$. Then there exists a constant $c = c(\mathbb{A}, \Delta_2(\psi), \nabla_2(\psi)) > 0$ such that 
		\begin{align}\label{eq:korngun}
			\int_{\R^n} \psi(|\mathrm{D}u|) \, \dd{x} \leq c \int_{\R^n} \psi(|\mathbb{A}u|) \, \dd{x}
		\end{align}
		holds for all $u \in \hold_c^\infty(\R^n; V)$.
	\end{satz}
	This result can be established by use of Theorem \ref{thm:avg} below, see \cite{contiFXG}. Moreover, we note that the constant $c > 0$ can be chosen to depend monotonically on $\Delta_2(\psi) + \nabla_2(\psi)$, which will be crucial for our subsequent applications. Moreover, we note that \eqref{eq:korngun} persists for domains $\Omega\subset\R^{n}$ and competitors $u\in\hold_{c}^{\infty}(\Omega;V)$ (by extending $\hold_{c}^{\infty}(\Omega;V)$-maps to $\R^{n}$ by zero); for related Korn-type inequalities on Orlicz spaces without zero boundary values, the reader is e.g. referred to \cite{breit_sharp_2012,DieningGmeineder,diening_decomposition_2010}.
	\subsection{Calderón-Zygmund operators} We now collect some basic assertions from Calderón-Zygmund theory, cf. \cite{duoandikoetxea_fourier_2001}.
	A Fourier multiplier operator $T_m : \lebe^2(\mathbb{R}^n; \mathbb{C}) \to \lebe^2(\mathbb{R}^n; \mathbb{C})$ is defined for $m \in \lebe^\infty(\mathbb{R}^n; \mathbb{C})$ by
	\[
	T_m f := \mathscr{F}^{-1} [ m \hat{f} ].
	\]
	A Fourier multiplier operator $T_m$ is a kernel operator if there exists a $k \in \lebe^1_{\text{loc}}(\mathbb{R}^n \setminus \{0\}; \mathbb{C})$ such that
	\[
	T_m(f)(x) = (k \ast f)(x)
	\]
	holds for all $x \not \in \operatorname{supp}(f)$.
	Moreover, a kernel operator $T_m$ satisfies the Hörmander condition if the map
	\[
	\mathbb{R}^n \ni y \mapsto \int \limits_{\{x \in \mathbb{R}^n : |x| > 2|y|\}} |k(x-y) - k(x)| \, \dd{x}
	\]
	is uniformly bounded. We then have the following result:
	\begin{satz}[Calderón-Zygmund, cf. \cite{calderon_existence_1952}] \label{thm:Calderon}
		Let $m \in \mathrm{L}^\infty(\R^n;\mathbb{C})$ and let $T_m$ be a kernel operator that satisfies the H\"ormander condition. Then $T_m$ extends to a bounded linear operator $T_m : \lebe^p(\mathbb{R}^n ;\mathbb{C}) \to \lebe^p(\mathbb{R}^n ;\mathbb{C})$ for $1 < p < \infty$.
	\end{satz}
	If $k$ is such that an $m$ exists with $T_m(f)(x) = (k \ast f)(x)$ for all $ x\not \in \operatorname{supp}(f)$ and this $m$ satisfies the conditions of Theorem \ref{thm:Calderon}, then we say that $m$ is a Calderón-Zygmund multiplier. For later applications, we also record the following theorem:  
	\begin{satz} [{{\cite[Theorem 1]{bagby_rearranged_1986_2}}}]
		\label{thm:bagby}
		Let $T_m$ be as in Theorem \ref{thm:Calderon} and let $f \in \lebe^p(\R^n; \mathbb{C})$ for some $p>1$. Then there is a $C>0$, such that \[
		(T_mf)^*(t) \leq C \left( \frac{1}{t} \int \limits_0^t f^*(s) \dd{s} + \int \limits_t^\infty f^*(s) \frac{\dd{s}}{s} \right),
		\]
		where for a measurable function $u$ $$u^*(s) := \sup \{ t \geq 0: |\{x \in \Omega: |u(x)| > t\}| > s \}$$ denotes the \emph{decreasing rearrangement} of $u$ for $s > 0$. 
	\end{satz}
	Finally, we record how Fourier multiplier operators can be expressed as Cauchy principal value singular integrals: 
	\begin{satz}[{{\cite[Theorem 4.13]{duoandikoetxea_fourier_2001}}}] \label{thm:avg}
		Let $m \in \hold^\infty(\R^n \backslash \{0\};\mathbb{C})$ be homogeneous of degree $0$. Then there exists $a \in \mathbb{C}$ and a function $\Omega \in \hold^\infty(\mathbb{S}^{n-1};\mathbb{C})$ with zero average such that   \[
		T_m f = af + \mathrm{p.v.} \frac{\Omega(x')}{|x|^n} \ast f\qquad\text{for any}\;f\in\mathscr{S}(\R^{n}), 
		\]
		where $x':=\frac{x}{|x|}$. Here, $\mathrm{p.v.}$ denotes the Cauchy principal value as usual. 
	\end{satz}
	\section{On $\varphi$-strong $\mathbb{A}$-quasiconvexity} \label{sec:quasiconvex}
	In this section, we discuss the $\varphi$-strong $\mathbb{A}$-quasiconvexity assumption as it appears in Theorem \ref{OrliczMainThm}, see \ref{item:F3}, and give a convenient reformulation. Indeed, the latter clarifies the meaning of $\varphi$-strong $\mathbb{A}$-quasiconvexity as a quasiconvex variant of $\varphi$-strong convexity. Moreover, we give connections of this notion to the coercivity of variational integrals \eqref{eq:functionalform} on suitable Dirichlet classes, which have been established first in the gradient case by \textsc{Chen \& Kristensen} \cite{chen_coercive}. 
	\subsection{Coercivity and existence of minimizers}
	Let $\mathbb{A}$ be an elliptic differential operator of the form \eqref{eq:diffop}, and let $\varphi$ be a Young function. For an open and bounded subset $\Omega\subset\R^{n}$, we define \[
	\sobo^{\mathbb{A}, \varphi}(\Omega) := \{u \in \lebe^\varphi(\Omega;V): \mathbb{A}u \in \lebe^\varphi(\Omega;W)\}, 
	\]
	which is endowed with the norm \[
	\norm{u}_{\sobo^{\mathbb{A}, \varphi}(\Omega) } := \norm{u}_{\lebe^\varphi(\Omega)} + \norm{\mathbb{A}u}_{\lebe^\varphi(\Omega)}. 
	\]
	We then define $\sobo^{\mathbb{A}, \varphi}_0(\Omega)$ as the closure of $\hold_c^\infty(\Omega;V)$ with respect to $\norm{\cdot}_{\sobo^{\mathbb{A}, \varphi}(\Omega;V) } $.

	For $u_{0}\in\sobo^{\mathbb{A},\varphi}(\Omega)$, we then introduce the corresponding Dirichlet class 
	\begin{equation}\label{eq:DirClass}
		\sobo_{u_{0}}^{\mathbb{A},\varphi}(\Omega):=u_{0}+\sobo_{0}^{\mathbb{A},\varphi}(\Omega). 
	\end{equation}
	Based on Remark \ref{rem:choiceofW}, it is no loss of generality to assume that $W=\mathscr{R}(\mathbb{A})\subset V\times\R^{n}$, and we will do so in the sequel. Our main objective in this section is to establish the coercivity and the existence of minimizers of the functional $\mathcal{F}$ on the Dirichlet classes \eqref{eq:DirClass}. To this end, we need two auxiliary results, and first collect an inequality of Poincar\'{e}-type as follows: 
	\begin{lemma}[of Poincar\'{e}-type]\label{lem:poincareinequality}
		Let $\Omega\subset\R^{n}$ be open and bounded. Moreover, let $\varphi\in\Delta_{2}$ and suppose that $\mathbb{A}$ is an elliptic first order differential operator of the form \eqref{eq:formA}. Then there exists a constant $c=c(\Omega,\mathbb{A},\Delta_{2}(\varphi))>0$ such that 
		\begin{align}\label{eq:PoincareMain}
			\int_{\Omega}\varphi(|\zeta|)\dif x \leq c\int_{\Omega}\varphi(|\mathbb{A}\zeta|)\dif x \qquad\text{for all}\;\zeta\in\sobo_{0}^{\mathbb{A},\varphi}(\Omega). 
		\end{align}
	\end{lemma}
	\begin{proof}
		The proof of this lemma is a variant of the symmetric gradient convolution-type Poincar\'{e} inequality from \cite[Prop. 5.1]{gmeineder_regularity_2020}.
		First let $\zeta\in\hold_{c}^{\infty}(\Omega;V)$, which we think to be extended to $\R^{n}$ by zero. Since $\mathbb{A}$ is elliptic, $\mathbb{A}[\xi]$ is injective for each $\xi\neq 0$. In particular, $\mathbb{A}^{*}[\xi]\mathbb{A}[\xi]$ is bijective for all $\xi\neq 0$, and so
		\begin{align*}
			\zeta = \mathscr{F}^{-1}\big((\mathbb{A}[\xi]^{*}\mathbb{A}[\xi])^{-1}\mathbb{A}^{*}[\xi]\mathscr{F}(\mathbb{A}\zeta)\big).
		\end{align*}
		The latter expression can be written as a convolution integral 
		\begin{align}\label{eq:PaulRep}
			\zeta(x) = \int_{\R^{n}}\mathfrak{K}(x-y)\mathbb{A}\zeta(y)\dif y,\qquad x\in\R^{n}, 
		\end{align}
		where $\mathfrak{K}\in\hold^{\infty}(\R^{n}\setminus\{0\};\mathrm{Lin}(W;V))$ is homogeneous of degree $(1-n)$ together with 
		\begin{align*}
			0<\inf_{x\neq y}\frac{|\mathfrak{K}(x-y)|}{|x-y|^{1-n}}\leq \sup_{x\neq y}\frac{|\mathfrak{K}(x-y)|}{|x-y|^{1-n}}<\infty.
		\end{align*}
		
		{Let $K(z) := \|\mathfrak{K}(z)\|_{\mathrm{op}}$. Since $\mathfrak{K}$ is homogeneous, $K(z) \simeq |z|^{1-n}$.  Let \[
			M(x) := \int \limits_\Omega K(x-y) \dd{y}.
			\]
			Since $\Omega$ is bounded, $M(x)$ is finite for all $x \in \Omega$, and $M$ is bounded from above and below: \begin{align*}
				0 < m' := \inf_{x\in\Omega} M(x) \leq \sup_{x\in\Omega} M(x) =: M < \infty.
			\end{align*}
			We define the probability measure $\nu_x$ on $\Omega$ by \[
			\dd \nu_x(y) := \frac{K(x-y)}{M(x)} \dd y.
			\]
			Therefore, and since $\zeta$ is supported on $\Omega$, \begin{align*}
				|\zeta(x)| & = \left| \int \limits_{\mathbb{R}^n} \mathfrak{K}(x-y) \mathbb{A} \zeta(y) \, \dd{y} \right| \leq M(x) \int \limits_\Omega  \frac{K(x-y)}{M(x)} |\mathbb{A} \zeta(y)| \dd{y} \\
				& = M(x) \int \limits_\Omega |\mathbb{A} \zeta(y) | \dd{\nu_x(y)}.
			\end{align*}
			Let $m \in \mathbb{N}$ be the smallest integer such that $M \leq 2^m$. By the $\Delta_2$-condition of $\varphi$, we obtain \begin{align*}
				\varphi(|\zeta(x)|) & \leq \varphi\left( 2^m \int_{\Omega} |\mathbb{A}\zeta(y)| \dif\nu_x(y) \right)\leq \Delta_{2}(\varphi)^{m} \cdot \varphi\left( \int_{\Omega} |\mathbb{A}\zeta(y)| \dif\nu_x(y) \right) \\
				& \leq \Delta_{2}(\varphi)^{m} \int_{\Omega} \varphi(|\mathbb{A}\zeta(y)|) \dif\nu_x(y),
			\end{align*}
			the latter by Jensen's inequality. Integrating this over $\Omega$ yields \begin{align*}
				\int_{\Omega}\varphi(|\zeta(x)|)\dif x &\leq \Delta_{2}(\varphi)^{m} \int_{\Omega} \int_{\Omega} \varphi(|\mathbb{A}\zeta(y)|) \frac{K(x-y)}{M(x)} \dd{y} \dd{x} \\
				&\leq \frac{\Delta_{2}(\varphi)^{m}}{m'} \int_{\Omega} \varphi(|\mathbb{A}\zeta(y)|) \left( \int_{\Omega} K(x-y) \dd{x} \right) \dif y \\
				& \leq \left(\frac{M\cdot \Delta_{2}(\varphi)^{m}}{m'}\right) \int_{\Omega} \varphi(|\mathbb{A}\zeta(y)|) \dif y.
			\end{align*}
		}
		From here, \eqref{eq:PoincareMain} follows by smooth approximation. 
	\end{proof}
	For the proof of Theorem \ref{lem:exis} below, we moreover require the following auxiliary ingredient; its proof is a direct adaptation of the arguments given in \cite[Lemma 5.2]{giusti2003direct}.
	\begin{lemma}[Lipschitz-type bound]\label{lem:Lipbound}
		Let $F\in\hold(\mathscr{R}(\mathbb{A}))$ be \emph{$\mathbb{A}$-rank-one convex,} meaning that 
		\begin{align*}
			t\mapsto F(z+ta\otimes_{\mathbb{A}}b),\qquad t\in[0,1] 
		\end{align*}
		is convex for all $z\in\mathscr{R}(\mathbb{A})$, $a\in V$ and $b\in\R^{n}$. If there exists an $N$-function $\varphi$ of class $\Delta_{2}$ such that 
		\begin{align*}
			|F(z)|\lesssim(1+\varphi(|z|))\qquad \text{for all}\; z\in\mathscr{R}(\mathbb{A}), 
		\end{align*}
		then we have the estimate 
		\begin{align*}
			|F(z)-F(w)| \lesssim \frac{\varphi(1+|z|+|w|)}{1+|z|+|w|}|z-w|\qquad\text{for all}\;z,w\in\mathscr{R}(\mathbb{A}). 
		\end{align*}
	\end{lemma}
	We now come to the main result of the present subsection:
	\begin{satz}[Coercivity and existence of minimizers] \label{lem:exis}
		Let $\mathbb{A}$ be an elliptic differential operator of order one, and let $\varphi\in\Delta_{2} \cap \nabla_2$ be an $N$-function; in particular, we have  \begin{align}\label{eq:lowerboundsuperlinear}
			\liminf_{t\to\infty}\frac{\varphi(t)}{t} = \infty.  
		\end{align}
		Under the conditions of Theorem \ref{OrliczMainThm} and assuming that $u_{0}\in\sobo^{\mathbb{A},\varphi}(\Omega)$, the variational integral $\mathcal{F}$ as in \eqref{eq:functionalform} is \emph{coercive} on $u_{0}+\sobo_{0}^{\mathbb{A},\varphi}(\Omega)$. In particular, for every $u_0 \in \sobo^{\mathbb{A},\varphi}(\Omega)$ there exists a minimizer $u$ of $\mathcal{F}$ over the affine Dirichlet class $u_0 + \sobo_0^{\mathbb{A}, \varphi}(\Omega)$, meaning that $u$ satisfies 
		\begin{align}\label{eq:minmain}
			\int_{\Omega}F(\mathbb{A}u)\dif x \leq \int_{\Omega}F(\mathbb{A}v)\dif x\qquad\text{for all}\;v\in u_{0}+\sobo_{0}^{\mathbb{A},\varphi}(\Omega).
		\end{align}
		In particular, the set of local minimizers is non-empty.
	\end{satz}
	\begin{proof} We employ the direct method of the Calculus of Variations, and first establish that $\mathcal{F}$ is coercive on Dirichlet classes $u_{0}+\sobo_{0}^{\mathbb{A},\varphi}(\Omega)$. To this end, we let $\zeta\in\sobo_{0}^{\mathbb{A},\varphi}(\Omega)$ be arbitrary. We put $u:=u_{0}+\zeta$ and split 
		\begin{align}\label{eq:Paulsplit}
			\begin{split}
				\int_{\Omega}F(\mathbb{A}u)\dif x & = 
				\int_{\Omega}F(\mathbb{A}(u_{0}+\zeta))\dif x \\ & = \Big(\int_{\Omega}F(\mathbb{A}(u_{0}+\zeta))\dif x - \int_{\Omega}F(\mathbb{A}\zeta)\dif x\Big) + \int_{\Omega}F(\mathbb{A}\zeta)\dif x \\ & =: \mathrm{I} + \mathrm{II}.
			\end{split}
		\end{align}
		For term $\mathrm{I}$, we first note that the definition of the convex conjugate (see \eqref{eq:Fenchel}) gives us for all $0<\varepsilon<1$ and all $t\geq 0$: 
		\begin{align}\label{eq:phi*bound}
			\varphi^{*}(\varepsilon t) := \sup_{s>0}s(\varepsilon t) - \varphi(s) \leq \sup_{s>0}\varepsilon(st - \varphi(s)) = \varepsilon \varphi^{*}(t).
		\end{align}
		With $0<\varepsilon<1$ to be fixed later, we now employ the definition of the convex conjugate in conjunction with Lemma \ref{lem:Lipbound} to find with a constant $c>0$ that 
		\begin{align*}
			|\mathrm{I}| & \lesssim  \int_{\Omega} \Big(\varepsilon\frac{\varphi(1+|\mathbb{A}u_{0}|+|\mathbb{A}\zeta|)}{1+|\mathbb{A}u_{0}|+|\mathbb{A}\zeta|}\Big) \Big(\frac{1}{\varepsilon}|\mathbb{A}u_{0}|\Big)\dif x \\ 
			& \lesssim  \int_{\Omega} \varphi^{*}\Big(\varepsilon\frac{\varphi(1+|\mathbb{A}u_{0}|+|\mathbb{A}\zeta|)}{1+|\mathbb{A}u_{0}|+|\mathbb{A}\zeta|}\Big)\dif x + \int_{\Omega}\varphi \Big(\frac{1}{\varepsilon}|\mathbb{A}u_{0}|\Big)\dif x \\ 
			& \lesssim  \varepsilon\int_{\Omega} \varphi^{*}\Big(\frac{\varphi(1+|\mathbb{A}u_{0}|+|\mathbb{A}\zeta|)}{1+|\mathbb{A}u_{0}|+|\mathbb{A}\zeta|}\Big)\dif x + \int_{\Omega}\varphi \Big(\frac{1}{\varepsilon}|\mathbb{A}u_{0}|\Big)\dif x \\ 
			& \lesssim  \varepsilon\int_{\Omega} \varphi(1+|\mathbb{A}u_{0}|+|\mathbb{A}\zeta|)\dif x + \int_{\Omega}\varphi \Big(\frac{1}{\varepsilon}|\mathbb{A}u_{0}|\Big)\dif x =: \mathrm{I}', 
		\end{align*}
		where the constants implicit in '$\lesssim$' do not depend on $\varepsilon$, $u_{0}$ or $\zeta$. Since $\varphi$ is of class $\Delta_{2}$, we conclude that 
		\begin{align}\label{eq:ultimo}
			|\mathrm{I}| \lesssim \mathrm{I}' \leq c_{1}\Big( \varepsilon\Big(|\Omega| + \int_{\Omega}\varphi(|\mathbb{A}u_{0}|)\dif x + \int_{\Omega}\varphi(|\mathbb{A}\zeta|)\dif x \Big) + \int_{\Omega}\varphi \Big(\frac{1}{\varepsilon}|\mathbb{A}u_{0}|\Big)\dif x\Big),  
		\end{align} 
		and as above, $c_{1}>0$ does not depend on $\varepsilon$, $u_{0}$ or $\zeta$.
		
		For term $\mathrm{II}$, we first note that there exist constants $d_{1},d_{2}\geq 1$ such that 
		\begin{align*}
			\frac{1}{d_{1}}\varphi(|z|)-d_{2} \leq \varphi_{2}(|z|)\qquad\text{for all}\;z\in \mathscr{R}(\mathbb{A}). 
		\end{align*}
		Using the $\varphi$-strong $\mathbb{A}$-quasiconvexity at zero and a routine technique (which lets us pass from the definition of $\mathbb{A}$-quasiconvexity on cubes to that on general domains $\Omega$), we therefore find that 
		\begin{align}\label{eq:IIest}
			\begin{split}
				\mathrm{II} & = \int_{\Omega}F(\mathbb{A}\zeta)-F(0)\dif x + F(0)|\Omega| \stackrel{\ref{item:F3}}{\geq} \nu_{1}\int_{\Omega}\varphi_{2}(|\mathbb{A}\zeta|)\dif x + F(0)|\Omega| \\ 
				& \geq \nu_{1}\int_{\Omega}\frac{1}{d_{1}}\varphi(|\mathbb{A}\zeta|)-d_{2}\dif x + F(0)|\Omega|.
			\end{split}
		\end{align}
		At this stage, we choose $m\in\mathbb{N}$ so large such that $\varepsilon:=2^{-m}$ satisfies 
		\begin{align*}
			0<2^{-m}=:\varepsilon \stackrel{!}{<} \frac{\nu_{1}}{4\Delta_{2}(\varphi)c_{1}d_{1}}. 
		\end{align*}
		On the other hand, since $\varphi$ is of class $\Delta_{2}$, we arrive at 
		\begin{align*}
			\varphi(|\mathbb{A}\zeta|) & \leq \varphi(|\mathbb{A}(u_{0}+\zeta)| + |\mathbb{A}u_{0}|) \stackrel{\text{$\varphi$ convex}}{\leq} \frac{1}{2}\big( \varphi(2|\mathbb{A}(u_{0}+\zeta)|)+\varphi(2|\mathbb{A}u_{0}|)\big) \\ & \leq \Delta_{2}(\varphi)\big(\varphi(|\mathbb{A}(u_{0}+\zeta)|)+\varphi(|\mathbb{A}u_{0}|)\big).
		\end{align*}
		Therefore, again using the $\Delta_{2}$-assumption on $\varphi$, gives us 
		\begin{align}\label{eq:Paulsplit1}
			\begin{split}
				|\mathrm{I}| & \stackrel{\eqref{eq:ultimo}}{\leq} \frac{\nu_{1}}{4d_{1}\Delta_{2}(\varphi)}|\Omega| + \frac{\nu_{1}}{4d_{1}\Delta_{2}(\varphi)}\int_{\Omega}\varphi(|\mathbb{A}u_{0}|)\dif x +  \frac{\nu_{1}}{4d_{1}}\int_{\Omega}\varphi(|\mathbb{A}u_{0}|)\dif x \\ & + \frac{\nu_{1}}{4d_{1}}\int_{\Omega}\varphi(|\mathbb{A}u|)\dif x +c_{1}\Delta(\varphi)^{m}\int_{\Omega}\varphi(|\mathbb{A}u_{0}|)\dif x.
			\end{split}
		\end{align}
		In consequence, merging this estimate with \eqref{eq:Paulsplit} and \eqref{eq:IIest}, we arrive at 
		\begin{align*}
			\int_{\Omega}F(\mathbb{A}u)\dif x & \geq \mathrm{II} - |\mathrm{I}| \geq c(\Delta_{2}(\varphi),m,u_{0},|\Omega|,F) + \frac{3\nu_{1}}{4d_{1}}\int_{\Omega}\varphi(|\mathbb{A}u|)\dif x, 
		\end{align*}
		where $c(\Delta_{2}(\varphi),m,u_{0},|\Omega|,F)\in\R$ is independent of $\zeta$.
		
		This estimate tells us that $\mathcal{F}$ is bounded below and coercive  on $u_{0}+\sobo_{0}^{\mathbb{A},\varphi}(\Omega)$ in the sense that 
		\begin{align*}
			((u_{j})\subset u_{0}+\sobo_{0}^{\mathbb{A},\varphi}(\Omega)\;\;\;\text{satisfies}\;\;\;\limsup_{j\to\infty}\|u_{j}\|_{\sobo^{\mathbb{A},\varphi}(\Omega)}=\infty)\Rightarrow \limsup_{j\to\infty}\mathcal{F}[u_{j}]=\infty. 
		\end{align*}
		Hence, $\mathfrak{m}:=\inf_{u_{0}+\sobo_{0}^{\mathbb{A},\varphi}}\mathcal{F}\in\R$, and we find a sequence $(u_{j})\subset u_{0}+\sobo_{0}^{\mathbb{A},\varphi}(\Omega)$ such that $\mathcal{F}[u_{j}]\to\mathfrak{m}$. By the Poincar\'{e}-type inequality from Lemma \ref{lem:poincareinequality}, we infer that $(u_{j})$ is bounded in $(\sobo^{\mathbb{A},\varphi}(\Omega),\|\cdot\|_{\sobo^{\mathbb{A},\varphi}(\Omega)})$. By our choice of $\varphi\in\Delta_{2}\cap\nabla_{2}$, this space is reflexive. Hence, there exists a (not relabeled) subsequence $(u_{j})$ and an element $u\in\sobo^{\mathbb{A},\varphi}(\Omega)$ such that $u_{j}\rightharpoonup u$ in $\sobo^{\mathbb{A},\varphi}(\Omega)$. The Dirichlet class $u_{0}+\sobo_{0}^{\mathbb{A},\varphi}(\Omega)$ is convex and closed with respect to the norm topology, and so it is weakly closed. Therefore, $u\in u_{0}+\sobo_{0}^{\mathbb{A},\varphi}(\Omega)$. In conclusion, the minimality of $u$ follows if we can show that $\mathcal{F}$ is lower semicontinuous with respect to weak convergence on Dirichlet subclasses of $\sobo^{\mathbb{A},\varphi}(\Omega)$. 
		
		To this end, recall that $W=\mathscr{R}(\mathbb{A})$. We consider $G(z):=F(\pi_{\mathbb{A}}(z))$, where $\pi_{\mathbb{A}}$ is a linear map such that $\mathbb{A}v=\pi_{\mathbb{A}}(\nabla v)$ for all $v\in\hold^{\infty}(\R^{n};V)$. Since $\varphi\in\Delta_{2}\cap\nabla_{2}$, Theorem \ref{thm:korn_uni} implies that $\sup_{j\in\mathbb{N}}\|\varphi(|\nabla u_{j}|)\|_{\lebe^{1}(\Omega)}<\infty$. Since $|G(z)|\lesssim 1+\varphi(|z|)$ for all $z\in W$, \cite[Prop. 10.9]{gmeineder_quasiconvex_2024} gives us
		\begin{align*}
			\int_{\Omega}G(\nabla u)\dif x \leq \liminf_{j\to\infty}\int_{\Omega}G(\nabla u_{j})\dif x,
		\end{align*}
		and this directly yields the requisite lower semicontinuity of $\mathcal{F}$. Summarizing, we obtain 
		\begin{align*}
			\mathcal{F}[u] \leq \liminf_{j\to\infty}\mathcal{F}[u_{j}] = \mathfrak{m}, 
		\end{align*}
		and so $u$ is a minimizer. This immediately implies that $u$ is a local minimizer, and so the claim follows.
	\end{proof}
	We end this subsection with two remarks.
	\begin{bem}
		If $\varphi\in\Delta_{2}\cap\nabla_{2}$, then $\sobo_{0}^{1,\varphi}(\Omega;V)=\sobo_{0}^{\mathbb{A},\varphi}(\Omega)$. Hence, in this case, the preceding result and its proof can be directly read as a coercivity and existence theorem on the classical {Orlicz-Sobolev} spaces $\sobo^{1,\varphi}(\Omega;V)$. Note however that the stronger statement $\sobo^{\mathbb{A},\varphi}(\Omega)=\sobo^{1,\varphi}(\Omega;V)$ only holds provided $\mathbb{A}$ satisfies the stronger condition of $\mathbb{C}$-ellipticity. 
	\end{bem}
	\begin{bem}
		It is also possible to generalize Theorem \ref{lem:exis} to $\varphi(t)=t\log(1+t)$,  even though this does not follow by reduction to full gradient functionals. {This, however, can} be achieved by employing a trace-preserving operator as in \cite[Prop. 10.9]{gmeineder_quasiconvex_2024} for $\mathbb{C}$-elliptic operators, and we shall pursue this in a future work. 
	\end{bem}
	\subsection{Equivalent definitions of $\varphi$-strong quasiconvexity}
	Usually, the definition of strong $\varphi$-quasiconvexity is not given by the gradient version of \ref{item:F3}, but rather by the following condition: For each $M>0$ there exists $\nu_{M}>0$ such that 
	\begin{equation}
		\int_{(0,1)^n} \left[ F(z_0 + \nabla \zeta) - F(z_0) \right] \dd{x} \geq \nu_{M} \label{eq:q2} \int_{(0,1)^n} V_\varphi(z_0 + \nabla \zeta) - V_\varphi(z_0)
	\end{equation}
	holds for all $|z_{0}|\leq M$ and all $\zeta\in\hold_{c}^{\infty}((0,1)^{n};V)$. Here, we work with the auxiliary $V$-function
	$$V_\varphi(z) := \varphi\left(\sqrt{1+|z|^2} \right) - \varphi(1),\qquad z\in X,$$ whenever $(X,|\cdot|)$ is a normed space (see also, for example, \cite{ carozza_partial_1998, contiFXG, Irving2021partial, schmidt1, schmidt_regularity_2009}). This reformulation asserts that, for each $M>0$ and each $z_{0}\in\R^{N\times n}$ with $|z_{0}|\leq M$, there exists $\nu_{M}>0$ such that 
	\begin{align}\label{eq:reformulation1}
		z\mapsto F(z)-\nu_{M}V_{\varphi}(z)\qquad \text{is quasiconvex at $z_{0}$}. 
	\end{align}
	We will establish in Theorem \ref{lem:quasiconvex} that both definitions are equivalent. However, \eqref{eq:reformulation1} tells us that strong $\varphi$-quasiconvexity can be read as the quasiconvex analogue of strong convexity for integrands of $\varphi$-growth; the latter corresponds to requiring that, for each $M>0$ and $z_{0}\in\R^{N\times n}$ with $|z_{0}|\leq M$, there exists $\nu_{M}>0$ such that 
	\begin{align*}
		z\mapsto F(z)-\nu_{M}V_{\varphi}(z)\qquad\text{is convex at $z_{0}$}.
	\end{align*} 
	The following lemma yields the equivalence of the two aforementioned versions of strong $\varphi$-quasiconvexity: 
	\begin{satz}[Equivalence of $\varphi$-strong quasiconvexity conditions] \label{lem:quasiconvex}
		Let $M>0$ and $\varphi \in \Delta_2 \cap \nabla_2$. For all $|z_0| \leq M$ and $\zeta \in \hold_c^\infty((0,1)^n; V)$ it holds
		\[
		\int_{(0,1)^n} V_\varphi(z_0 + \nabla \zeta) - V_\varphi(z_0) \, \mathrm{d}x \simeq  \int_{(0,1)^n} \varphi_{1+|z_0|}(|\nabla \zeta|) \, \mathrm{d}x,
		\]
		where here and in what follows the underlying constants depend only on $M$ and $\varphi$. Similarly, we have that
		\[
		\int_{(0,1)^n} V_\varphi(z_0 + \mathbb{A}\zeta) - V_\varphi(z_0) \, \mathrm{d}x \simeq  \int_{(0,1)^n} \varphi_{1+|z_0|}(|\mathbb{A} \zeta|) \, \mathrm{d}x, 
		\]
		and the underlying constants again only depend on $M$ and $\varphi$.
	\end{satz}
	For the proof of Theorem \ref{lem:quasiconvex}, we record three auxiliary lemmas. We begin with: 
	\begin{lemma} \label{lem:aux1}
		Let $\varphi: \R_{\geq 0} \to \R_{\geq 0}$ be strictly monotonically increasing with either 
		\begin{enumerate}
			\item  $\varphi \in \Delta_2 \cap \nabla_2$ or
			\item $\varphi(t)=t\log(1+t)$, $t\geq 0$, 
		\end{enumerate} 
		and let $M>0$ be given. 
		Then there exists a  constant $S>3M$ and $C_1,C_2>0$ such that for all $\xi$ with $|\xi|\geq S$ and all $z_0$ with $|z_0| \leq M$ it holds 
		\begin{align}
			C_1 \left( \varphi(\sqrt{1+|z_0+\xi|^2}) \right.&\left. - \varphi(\sqrt{1+|z_0|^2}) \right) \notag \\ & \leq \varphi(1+|z_0|+|\xi|) - \varphi(1+|z_0|) \label{eq:aux1.1} \\
			& \leq C_2 \left( \varphi(\sqrt{1+|z_0+\xi|^2}) - \varphi(\sqrt{1+|z_0|^2}) \right). \label{eq:aux1.2}
		\end{align}
	\end{lemma}
	\begin{proof}[Proof of Lemma \ref{lem:aux1}] 
		For \eqref{eq:aux1.1}, we show that there exists $C_1>0$ with 
		\begin{align}\label{eq:helpful0A}
			C_1 \leq \frac{\varphi(1+|z_0|+|\xi|) - \varphi(1+|z_0|)}{\varphi(\sqrt{1+|z_0+\xi|^2}) - \varphi(\sqrt{1+|z_0|^2})} =: (*)
		\end{align}
		for all $|\xi|\geq S$ and $|z_{0}|\leq M$. 
		
		We have that 
		\begin{align}\label{eq:helpful1}
			\begin{split}
				\varphi(\sqrt{1+|z_{0}+\xi|^{2}})-\varphi(\sqrt{1+|z_{0}|^{2}}) & \leq \varphi(\sqrt{1+|z_{0}+\xi|^{2}}) \\ & \leq \varphi(1+|z_{0}|+|\xi|). 
			\end{split}
		\end{align}
		Now, since $\varphi$ is continuous and strictly monotonically increasing, there exists $S>M$ such that $|\xi|\geq S$ implies that 
		\begin{align*}
			2\varphi(1+M)\leq \varphi(1+|\xi|), 
		\end{align*}
		whereby $|z_{0}|\leq M$ and $|\xi|\geq S$ give us 
		\begin{align*}
			\varphi(1+|z_{0}|) \leq \varphi(1+M) \leq \frac{1}{2}\varphi(1+|\xi|)\leq \frac{1}{2}\varphi(1+|z_{0}|+|\xi|),
		\end{align*}
		so that in particular
		\begin{align}\label{eq:helpful2}
			-\frac{\varphi(1+|z_{0}|)}{\varphi(1+|z_{0}|+|\xi|)}\geq -\frac{1}{2}. 
		\end{align}
		Therefore, $|z_{0}|\leq M$ and $|\xi|\geq S$ combine to 
		\begin{align*}
			(*) & \stackrel{\eqref{eq:helpful1}}{\geq}  \frac{\varphi(1+|z_0|+|\xi|)}{\varphi(1+|z_{0}|+|\xi|)} - \frac{\varphi(1+|z_{0}|)}{\varphi(1+|z_{0}|+|\xi|)} \stackrel{\eqref{eq:helpful2}}{\geq} \frac{1}{2} =: C_1.
		\end{align*}
		This establishes \eqref{eq:helpful0A}.
		For \eqref{eq:aux1.2} we proceed similarly, but have to argue slightly differently. We show that a $C_2 > 0$ exists with \[
		C_2 \leq  \frac{\varphi(\sqrt{1+|z_0+\xi|^2}) - \varphi(\sqrt{1+|z_0|^2})}{\varphi(1+|z_0|+|\xi|) - \varphi(1+|z_0|)} =: (**).
		\]
		We have \[
		|z_0| + |\xi| \leq M + |\xi| \leq 2M + |\xi| - |z_0| \leq 2|z_0 + \xi|.
		\]
		This gives us \[
		1 + |z_0| + |\xi| \leq 2 + 2|z_0 + \xi| \leq 12\sqrt{1+|z_0+\xi|^2}.
		\]
		Since $\sqrt{1+|z_0|^2} \leq 1+|z_0|$, we have \[
		-\varphi \left(\sqrt{1+|z_0|^2} \right)  \geq -\varphi(1+|z_0|)
		\]
		and therefore \[
		(**) \geq \frac{\varphi \left(\frac{1}{12}(1+|z_0|+|\xi| \right) - \varphi(1+|z_0|)}{\varphi(1+|z_0|+|\xi|)} =: (***).
		\]
		If now $\varphi \in \Delta_2 \cap \nabla_2$, then we find an $K>0$ such that \[
		\varphi \left(\frac{1}{12}(1+|z_0|+|\xi|)\right) \geq K\varphi(1+|z_0|+|\xi|).
		\] 
		If on the other hand $\varphi(t) = t \log(1+t)$, we define \[
		f(t) := \frac{\log(1+t)}{\log(1+\frac{1}{12}t)}. 
		\]
		For $t \geq 1$, $f(t) > 0$ holds. Furthermore, observe the asymptotic behavior of $f(t)$ as $t \to \infty$. Expanding the logarithmic terms for large $t$, we find
		\[
		\log(1+t) \sim \log(t), \quad \log\left(1 + \frac{1}{12}t\right) \sim \log\left(\frac{1}{12}t\right) = \log(t) - \log(12).
		\]
		Thus, 
		\[
		f(t) \to \frac{\log(t)}{\log(t) - \log(12)} = \frac{1}{1 - \frac{\log(12)}{\log(t)}}.
		\]
		As $t \to \infty$, $f(t) \to 1$. Since $f$ is continuous in $t$, we conclude that we could find a $K>0$, such that \[
		\frac{1}{12K} \leq f(t)
		\]
		for all $t \geq 1$ and therefore again \[
		\varphi \left(\frac{1}{12}(1+|z_0|+|\xi|) \right) \geq K\varphi(1+|z_0|+|\xi|).
		\]
		In both cases, $\varphi$ is strictly monotonically increasing, and we therefore find an $S>3M$, such that for all $|\xi| \geq S$ implies \[
		\frac{2}{K} \varphi(1+M) \leq \varphi(1+|\xi|),
		\]
		whereby $|z_0| \leq M$ and $|\xi| \geq S$ give us (in the same way as in \eqref{eq:helpful2}) \begin{equation} \label{eq:helpful3}
			-\frac{\varphi(1+|z_{0}|)}{\varphi(1+|z_{0}|+|\xi|)}\geq -\frac{K}{2}. 
		\end{equation}
		Now we conclude that \[
		(***) \geq \frac{K\varphi(1+|z_0|+|\xi|) - \varphi(1+|z_0|)}{\varphi(1+|z_0|+|\xi|)} \overset{\eqref{eq:helpful3}}{\geq} K - \frac{1}{2}K = \frac{1}{2} K =: C_2, 
		\]
		and this implies the claim. 
	\end{proof}
	
	\begin{lemma} \label{lem:aux2}
		Let $\varphi \in \hold^2((0, \infty), \mathbb{R}_{\geq 0})$ {be convex and strictly increasing}, and let $M>0$. For all $|z_0| \leq M$ and all $S>0$, there are constants $C_3, C_4 > 0$ such that 
		\[
		E(z_0, \xi) := \varphi(\sqrt{1+|z_0+\xi|^2}) - \varphi(\sqrt{1+|z_0|^2}) - \langle \nabla_\xi[\varphi(\sqrt{1+|z_0+\xi|^2})]_{\mid \xi=0}, \xi \rangle
		\]
		satisfies
		\begin{equation}
			C_3 |\xi|^2 \leq E(z_0, \xi) \leq C_4 |\xi|^2 \label{eq:bounds}
		\end{equation}
		for all $|\xi|\leq S$. The constants depend on $M$ and $S$.
	\end{lemma}
	\begin{proof}[{Proof of Lemma \ref{lem:aux2}}]
		{
			Let $f_{z_0}(\xi) := \varphi \left(\sqrt{1+|z_0 + \xi|^2} \right)$. The Taylor expansion of $f_{z_0}$ is \[
			f_{z_0} (\xi) =  \varphi(\sqrt{1+|z_0|^2}) + \langle \nabla_\xi[\varphi(\sqrt{1+|z_0+\xi|^2})]_{\mid \xi=0}, \xi \rangle + \frac{1}{2} \xi^\top H_{f_{z_0}} (\theta \xi) \xi,
			\]
			where $\theta \in (0,1)$, and $H_{f_{z_0}}$ is the Hessian of $f_{z_0}$. Therefore, \[
			E(z_0, \xi) = \frac{1}{2} \xi^\top H_{f_{z_0}}(\theta \xi) \xi. 
			\]
			Since $\varphi$ is a strictly increasing, convex function, and $\xi \mapsto \sqrt{1+|z_0 + \xi|^2}$ is strictly convex, $H_{f_{z_0}}$ is positive definite. By continuity, all eigenvalues of $H_{f_{z_0}} (\theta \xi)$ are uniformly bounded from below by some positive constant on the compact set, where $|z_0| \leq M$ and $|\xi| \leq S$. This yields \eqref{eq:bounds}.}
	\end{proof}

	\begin{lemma} \label{lem:quadr}
		Let $M,S>0$ and $|z_0| \leq M$. Let $\varphi$ be an $N$-function. Then \[
		\varphi_{1+|z_0|}(|z|) \simeq |z|^2 
		\]
		for all $|z_0| \leq M$ and $|z| \leq S$, where the underlying constants only depend on $M$ and $S$.
	\end{lemma}
	
	\begin{proof}
		Let $|z_0| \leq M$. It suffices to show that \[
		g(z_0,z):=\frac{|z|^2}{\varphi_{1+|z_0|}(|z|)}, \,\, |z| \leq S
		\]
		is bounded both above and below by positive constants that rely only on $M$ and $S$. Since $\varphi'_{1+|z_0|}(|z|) = 0$ if and only if $z=0$, we conclude that \[
		\lim \limits_{|z| \searrow 0}  \frac{|z|^2}{\varphi_{1+|z_0|}(|z|)} = \lim \limits_{|z| \searrow 0} \frac{2}{\varphi''_{1+|z_0|}(|z|)},
		\]
		so it suffices to show that $\varphi''_{1+|z_0|}(0) > 0$. Indeed, it holds \[
		\varphi''_{1+|z_0|}(|z|) = \frac{\varphi''(1+|z_0|+|z|)}{1+|z_0|+|z|}|z| + \varphi'(1+|z_0|+|z|)\frac{1+|z_0|}{(1+|z_0|+|z|)^2},
		\]
		and therefore \[
		\varphi''_{1+|z_0|}(0) = \frac{\varphi'(1+|z_0|)}{1+|z_0|} > 0.
		\]
		Now it remains to state that $g$ is continuous in both variables, implying that on the compact set $|z_0|\leq M$, $|z| \leq S$, it is therefore positively bounded from above and below.
	\end{proof}
	We now come to the: 
	\begin{proof}[{{Proof of Theorem \ref{lem:quasiconvex}}}]
		Let $S := \max \{3M, 1\}$ and $Q=(0,1)^n$. We decompose the integral as follows:
		\begin{align*}
			\int \limits_{Q} \varphi_{1+|z_0|}(|\nabla \zeta|) \, \mathrm{d}x 
			&= \int \limits_{Q \cap \{|\nabla \zeta| \leq S\}} \varphi_{1+|z_0|}(|\nabla \zeta|) \, \mathrm{d}x + \int \limits_{Q \cap \{|\nabla \zeta| \geq S\}} \varphi_{1+|z_0|}(|\nabla \zeta|) \, \mathrm{d}x \\
			&=: A + B.
		\end{align*}
		\noindent
		First, by Lemma \ref{lem:quadr} we have
		\begin{equation}
			\varphi_{1+|z_0|}(|z|) \simeq |z|^2 \label{shift:est}
		\end{equation}
		for $|z_0| \leq M$ and $|\nabla \zeta| \leq S$.
		By Lemma \ref{lem:aux2}, the Taylor remainder of $V_\varphi$ is also quadratically bounded. Since $\zeta \in \hold_c^\infty((0,1)^n;V)$, the integral over the linear term vanishes:
		\[
		\int_Q \langle (\nabla V_\varphi)(z_0), \nabla \zeta \rangle \dd{x} = 0.
		\]
		Thus, we arrive at
		\begin{align*}
			A & \lesssim \int \limits_{Q \cap \{|\nabla \zeta| \leq S\}} |\nabla \zeta|^2 \dd{x} \\
			& \stackrel{\text{Lemma}~\ref{lem:aux2}}{\lesssim} \int \limits_{Q \cap \{|\nabla \zeta| \leq S\}} V_\varphi(z_0 + \nabla \zeta) - V_\varphi(z_0) - \langle (\nabla V_\varphi)(z_0), \nabla \zeta \rangle \dd{x} \\
			& = \int \limits_{Q \cap \{|\nabla \zeta| \leq S\}} V_\varphi(z_0 + \nabla \zeta) - V_\varphi(z_0) \dd{x}.
		\end{align*}
		As the integrand is positive due to the convexity of $V_\varphi$, we can extend the integral to the whole domain $Q$:
		\[
		\int \limits_{Q \cap \{|\nabla \zeta| \leq S\}} V_\varphi(z_0 + \nabla \zeta) - V_\varphi(z_0) \dd{x} \le \int_Q V_\varphi(z_0 + \nabla \zeta) - V_\varphi(z_0) \dd{x}.
		\]
		For $B$, we use that $\varphi'(0)=0$ and therefore 
		\begin{align*}
			\varphi_{1+|z_0|}(|\nabla \zeta|)  = \int \limits_0^{|\nabla \zeta|} \frac{\varphi'(1+|z_0|+\tau)}{1+|z_0|+\tau} \tau \dd{\tau}  \simeq \int \limits_0^{|\nabla \zeta|} {\varphi''(1+|z_0|+\tau)} \tau \dd{\tau} =: \tilde{B}.
		\end{align*}
		Integrating by parts, we further get
		\begin{align*}
			\tilde{B}
			= \varphi'(1+|z_0|+|\nabla \zeta|)&|\nabla \zeta|  - \varphi(1+|z_0|+ |\nabla \zeta|) + \varphi(1+|z_0|) \\
			& \lesssim \varphi(1+|z_0|+ |\nabla \zeta|)  - \varphi(1+|z_0|),
		\end{align*}
		where for the last estimate we need that $\varphi(0)=0.$ Now, Lemma \ref{lem:aux1} implies that 
		\begin{equation}
			\varphi(1+|z_0|+|\nabla \zeta|) - \varphi(1+|z_0|) \lesssim \varphi(\sqrt{1+|z_0+\nabla \zeta|^2}) - \varphi(\sqrt{1+| z_0|^2}), \label{est}
		\end{equation}
		since $\varphi \in \Delta_2 \cap \nabla_2$.
		Thus, 
		\begin{align*}
			B & \lesssim \int \limits_{Q \cap \{|\nabla \zeta| \geq S\}} \varphi \left(\sqrt{1+|z_0+\nabla \zeta|^2} \right) - \varphi \left(\sqrt{1+|z_0|^2} \right) \\
			& = \int \limits_{Q \cap \{|\nabla \zeta| \geq S\}} V_\varphi(z_0 + \nabla \zeta) - V_\varphi(z_0) \\
			& \leq \int \limits_{Q} V_\varphi(z_0 + \nabla \zeta) - V_\varphi(z_0).
		\end{align*}
		We have thus estimated $A + B$ from above. 
		
		For the lower estimate, we proceed analogously. The integrands are positive, so we can estimate the terms individually. For term $A$, we use Lemma \ref{lem:aux2}.
		\begin{align*}
			A & \stackrel{\text{Lemma}~\ref{lem:quadr}}{\simeq} \int \limits_{Q \cap \{|\nabla \zeta| \leq S\}} |\nabla \zeta|^2 \dd{x} \\
			& \stackrel{\text{Lemma}~\ref{lem:aux2}}{\gtrsim} \int \limits_{Q \cap \{|\nabla \zeta| \leq S\}} \left( V_\varphi(z_0 + \nabla \zeta) - V_\varphi(z_0) - \langle (\nabla V_\varphi)(z_0), \nabla \zeta \rangle \right) \dd{x} =: C
		\end{align*}
		and \begin{align*}
			B & \simeq \int \limits_{Q \cap \{|\nabla \zeta| \geq S\}} \varphi(1+|z_0|+ |\nabla \zeta|)  - \varphi(1+|z_0|) \dd{x} \\
			& \gtrsim  \int \limits_{Q \cap \{|\nabla \zeta| \geq S\}} \varphi(1+|z_0+ \nabla \zeta|)  - \varphi(1+|z_0|) \dd{x} \\
			& \gtrsim \int \limits_{Q \cap \{|\nabla \zeta| \geq S\}} \varphi(\sqrt{1+|z_0+ \nabla \zeta|^2})  - \varphi(\sqrt{1+|z_0|^2}) \dd{x} \\
			& \simeq \int \limits_{Q \cap \{|\nabla \zeta| \geq S\}} V_\varphi(z_0+ \nabla \zeta)  - V_\varphi(z_0) \dd{x} =: D.
		\end{align*}
		Note that the integrands are positive due to convexity, and we used the same argument as in \eqref{est}, just the other way around.
		Since $$C + D = \int \limits_{Q} V_\varphi(z_0 + \nabla \zeta) - V_\varphi(z_0) \dd{x},$$ the claim follows for $\mathbb{A}=\nabla$. It is clear that we may replace $\nabla$ by $\mathbb{A}$ in the above estimates, and so the proof is complete.
	\end{proof}
	
	\section{Proof of Theorem \ref{OrliczMainThm}}
	\label{Proof}
	\subsection{Partial regularity for full gradient functionals}
	Following the strategy of \textsc{Conti \& Gmeineder} \cite{contiFXG,gmeineder_partial_2021}, the key to Theorem \ref{OrliczMainThm} is a reduction to the corresponding partial regularity results for full gradient functionals. Therefore, we first collect in this subsection the requisite regularity results which make this reduction procedure possible. 
	
	As above, we let $V,W$ be two finite dimensional, real inner product spaces. In the following, we assume that $H: W \to \R$ satisfies:
	\begin{enumerate}[(H1)]
		\item $H \in \hold^2(W)$, \label{H1}
		\item There exists a constant $c > 0$ such that $$|H(z)| \leq c(1+\varphi(|z|))$$ for all $z \in W$, \label{H2}
		\item For every $M > 0$, there exists a constant $l_M > 0$ such that
		\[
		\int_{(0,1)^n} \left[ H(z_0 + \nabla \zeta) - H(z_0) \right] \dd{x} \geq l_M \int_{(0,1)^n} \varphi_{1+|z_0|}(|\nabla \zeta|) \dd{x}
		\]
		for all $|z_0| < M$ and all  $\zeta \in \hold_c^\infty((0,1)^n; V)$. \label{H3}
	\end{enumerate}
	Our focus is on local minimizers of the functional
	\begin{equation}
		\Phi(v; \omega) := \int_\omega H(\mathrm{D} v) \dd{x}, \label{eq:functional}
	\end{equation}
	where $v \in \sobo^{1,\varphi}(\omega; V)$, and $\omega\subset\R^{n}$ is open and bounded. Given an open and bounded subset $\Omega\subset\R^{n}$, a function $u \in \sobo_\text{loc}^{1,\varphi}(\Omega; V)$ is termed a \emph{local minimizer of \eqref{eq:functional}} if
	\begin{equation}
		\Phi(u; \omega) \leq \Phi(u + \zeta; \omega) \label{loc:min}
	\end{equation}
	holds for every open and bounded set ${\omega} \Subset \Omega$ and all $\zeta \in \sobo_{0}^{1,\varphi}(\omega;V)$. 
	We then have the following partial regularity theorem: 
	\begin{satz}[{{Partial regularity theorem, cf. \cite[Theorem 1.2]{Irving2021partial}, \cite{diening_partial_2012}}}]
		\label{thm:IrvMainReg}
		Let $H$ satisfy the conditions \ref{H1}-\ref{H3}. For $\Omega \subset \mathbb{R}^n$ open and bounded, $\widetilde{M} > 0$, and  $\alpha \in (0,1)$, there exists $\varepsilon > 0$ such that the following holds: If $u \in \sobo_\mathrm{loc}^{1, \varphi}(\Omega; V)$ is a local minimizer of \eqref{eq:functional} and $B_R(x_0) \Subset \Omega$ is such that both $|(\nabla u)_{B_R(x_0)}| \leq \widetilde{M}$ and 
		\[
		\fint_{B_R(x_0)} \varphi_{1+\widetilde{M}} (|\nabla u - (\nabla u)_{B_R(x_0)}|) \, \mathrm{d}x \leq \varepsilon,
		\] 
		are satisfied, 
		then $u$ is of class $\hold^{1, \alpha}$ in $B_{\frac{R}{2}}(x_0)$.
	\end{satz}
	Theorem \ref{thm:IrvMainReg} represents just a small part of the theory of partial regularity for minimizers, but it is very useful for our purposes.
	\subsection{Proof of Theorem \ref{OrliczMainThm}}
	Let $\Omega \subset \mathbb{R}^n$ be open. Given an integrand $F: W \to \mathbb{R}$, we define
	\[
	\mathcal{F}[u;\omega] := \int_\omega F(\mathbb{A}u) \, \mathrm{d}x
	\]
	for all open subsets $\omega \subset \Omega$. 
	Based on the previous lemma, we may assume that $\mathbb{A}$ is elliptic. For future reference, we note that it is then no loss of generality to assume that 
	\begin{align}\label{eq:embass}
		W\hookrightarrow \R^{N\times n}, 
	\end{align}
	see Remark \ref{rem:choiceofW}. 
	
	We now proceed with the proof of Theorem \ref{OrliczMainThm}. The  technique of proof employed here is parallel to that described in \cite{contiFXG} for functions with power growth, particularly through the utilization of a reduction strategy.
	\begin{proof}[Proof of Theorem \ref{OrliczMainThm}]
		As in \cite{contiFXG}, to establish the implication $(a)\!\implies\!(b)$ in Theorem \ref{OrliczMainThm}, we aim to connect it with Theorem \ref{thm:IrvMainReg}. To achieve this, we introduce the variational integrand
		\begin{equation}
			G: V \times \mathbb{R}^n \to \R, \quad z \mapsto F(\pi_{\mathbb{A}}(z)), \label{eq:G}
		\end{equation}
		where $\pi_\mathbb{A}$ is a linear map such that $\mathbb{A}v(x) = \pi_\mathbb{A}(\nabla v)(x)$; this is possible due to \cite[Section 3]{contiFXG}, and is analogous to our argument in the proof of Theorem \ref{lem:exis}. Our strategy involves applying Theorem \ref{thm:IrvMainReg} to $G$. Hence, we need to establish the following properties:
		\begin{enumerate}
			\item By \ref{F1}, since $F \in \hold^2(W)$ and $\pi_{\mathbb{A}}$ is linear, we have $G \in \hold^2(V \times \mathbb{R}^n)$.
			\item By \ref{F2}, there exists some $c > 0$ such that $|F(z)| \leq c(1+\varphi(|z|))$ for all $z \in W$. Since $\pi_\mathbb{A}$ is a linear operator on a finite-dimensional space, $\pi_\mathbb{A}$ is bounded. Therefore, we conclude that because $\varphi \in \Delta_2$
			\[
			|G(z)| \leq c(1+\varphi(|z|))
			\]
			for all $z \in V \times \mathbb{R}^n$.
			
			\item Let $M > 0$. Then we have
			\begin{align}
				\int \limits_{\R^n} G(z_0 + \nabla \zeta) - G(z_0) \, \dd{x} & = \int \limits_{\R^n} F(\pi_{\mathbb{A}} (z_0) + \mathbb{A} \zeta) - F(\pi_{\mathbb{A}} (z_0)) \, \dd{x} \label{eq:juhu}\\
				& \geq  \nu_M \int \limits_{\R^n} \varphi_{1+|\pi_{\mathbb{A}}(z_0)|}(|\mathbb{A} \zeta|) \, \dd{x},\notag
			\end{align}
			by \ref{F3} for all $\zeta \in \hold_c^\infty((0,1)^n; V)$ and $|z_0| \leq M$.
			We note that $\Delta_2(\varphi_a) \simeq \Delta_2(\varphi)$ and $\nabla_2(\varphi_a) \simeq\nabla_2(\varphi)$ uniformly in $a$ for all $a>0$ because of Lemma \ref{Delta:Nabla:Comp}. Moreover, since $\pi_{\mathbb{A}}$ is bounded, we find a $\delta > 0$ such that $|\pi_{\mathbb{A}}(z)| \leq \delta |z|$ for all $z \in V$. Since $\varphi_b \simeq \varphi_{M}$ for all $0 \leq b \leq M$ with a constant only depending on $M$ by \eqref{shift:est}, we conclude that
			\[
			\nu_M \int \limits_{\R^n} \varphi_{1+|\pi_{\mathbb{A}}(z_0)|}(|\mathbb{A} \zeta|) \, \dd{x} \geq \nu_M \mu \int \limits_{\R^n} \varphi_{1+|z_0|}(|\mathbb{A} \zeta|) \, \dd{x},
			\]
			where $\mu$ only depends on $M$. We choose $c > 0$ as in Theorem \ref{thm:korn_uni}. Note that $c > 0$ does not depend on $z_0$, because it can be chosen monotonically in $\Delta_2(\varphi) + \nabla_2(\varphi)$. Then we have
			\[
			\nu_M \mu \int \limits_{\R^n} \varphi_{1+|z_0|}(|\mathbb{A} \zeta|) \, \dd{x} \geq c \nu_M \mu \int \limits_{\R^n} \varphi_{1+|z_0|}(|\nabla \zeta|) \, \dd{x}.
			\]
			Now we can choose $l_M := c \nu_M \mu$, which depends only on $M$ and $\varphi$.
		\end{enumerate}
		According to Theorem \ref{thm:IrvMainReg}, all local minima of $v \mapsto \int G(\mathrm{D} v) \, \dd{x}$ belong to the class $\hold_\text{loc}^{1, \alpha}$. Following the approach in \cite{contiFXG}, this implies that for any open subset $\omega \Subset \Omega$ and all $u \in \sobo^{1,p}(\Omega;V)$,
		\begin{equation}
			\mathcal{G}(u; \omega) := \int_\omega G(\mathrm{D}u) \dd{x} = \int_\omega F(\mathbb{A}u) \dd{x} = \mathcal{F}[u;\omega]. \label{fct:G}
		\end{equation}
		It follows that every local minimizer of $\mathcal{F}$ also minimizes $\mathcal{G}$. Thus, the $\hold_{\text{loc}}^{1, \alpha}$ regularity is transferred.
		
		For  $(b)\!\implies\!(a)$ we invoke Lemma \ref{Kernel} as in \cite{contiFXG}. Assuming $\mathbb{A}$ is not elliptic, Lemma \ref{Kernel} provides $v \in \ker(\mathbb{A})$ such that $v \not \in \hold(\omega; V)$ for all open $\omega \subset \Omega$. If $u$ is a local minimizer of class $\hold^{1,\alpha}_{\text{loc}}$, then $u+v$ is also a local minimizer but does not belong to $\hold^{1,\alpha}_{\text{loc}}$. This completes the proof
	\end{proof}
	\begin{bem}[Higher order operators]
		While our analysis primarily focuses on first order differential operators, the reduction techniques and regularity results naturally extend to higher order elliptic homogeneous differential operators of the form 
		\begin{align*}
			\mathbb{A}=\sum_{|\alpha|=k}\mathbb{A}_{\alpha}\partial^{\alpha}.
		\end{align*}
		We leave the straightforward modifications, which rely on \textsc{Irving}'s higher order variant of Theorem \ref{thm:IrvMainReg}, see \cite{Irving2021partial}, to the reader.
	\end{bem}
	\section{Generalization to $L \log L$-growth}\label{sec:LlogL}
	\subsection{A Korn-type inequality}
	Now we consider the case where $\varphi(t) := t \log (1+t)$, which implies $\varphi \in \Delta_2 \backslash \nabla_2$. Our strategy will be similar to that in the previous section. Therefore, we require a Korn inequality applicable to this case. To achieve this, we generalize \textsc{Cianchi's} result \cite[Theorem 3.1]{cianchi_korn_2014} for elliptic operators $\mathbb{A}$. This result builds on earlier work on strong and weak type inequalities for classical operators in Orlicz spaces, as discussed in \cite{cianchi2}. Similar results for the symmetric gradient have been achieved in \cite{breit_trace-free_2017}.
	\begin{satz}[Korn-type inequality in Orlicz spaces] \label{lem:cianchi_ver}
		Let $\Omega$ be an open and bounded set in $\R^n$, $n \geq 2$, and let $\mathbb{A}$ be a homogeneous, first order, elliptic differential operator of the form \eqref{eq:formA}. Moreover, let $\Psi$ and $\Phi$ be Young functions such that
		\begin{equation}
			t \int \limits_{0}^t \frac{\Phi(s)}{s^2} \, \dd{s} \leq \Psi(ct)  \label{3.1}
		\end{equation}
		and 
		\begin{equation}
			t \int \limits_{0}^t \frac{\Psi^*(s)}{s^2} \, \dd{s} \leq \Phi^*(ct) \label{3.2}
		\end{equation}
		hold for all $t \geq 0$ and some $c > 0$. Then there exists a constant $C > 0$ such that
		\[
		\int \limits_\Omega \Phi(|\nabla u|) \, \dd{x} \lesssim \int \limits_\Omega \Psi(C|\mathbb{A}u|) \, \dd{x}\qquad\text{for all}\; u \in \sobo_0^{1,\Psi}(\Omega;V).
		\]
	\end{satz}
	In particular, inequalities \eqref{3.1} and \eqref{3.2} imply that $\Psi$ dominates $\Phi$ globally (cf. \cite[Proposition 3.5]{cianchi_korn_2014}), which means that there is a $C>0$ such that \[
	\Psi(t) \leq \Phi(Ct)
	\]
	for all $t \geq 0$. For the proof of Theorem \ref{lem:cianchi_ver} we need the following lemma.
	\begin{lemma}[{{\cite[Lemmas 5.2 and 5.4]{cianchi_korn_2014}}}] \label{lem:cianchi_hardy} The following holds:
		\begin{enumerate}
			\item Assume that \eqref{3.1} holds. Then $(0,\infty)$  \[
			\int \limits_0^\infty \Phi \Big( \frac{1}{s} \int \limits_0^s u(r) \dd{r} \Big) \dd{s} \leq \int \limits_0^\infty \Psi( 2c  u(s)) \dd{s}
			\]
			holds for every nonnegative measurable function $u \in \sobo_0^{1,\Psi}(\Omega;V)$.
			\item Assume that \eqref{3.2} holds. Then $(0,\infty)$ \[
			\int \limits_0^\infty \Phi \Big(  \int \limits_s^\infty u(r) \frac{\dd{r}}{r} \Big) \dd{s} \leq \int \limits_0^\infty \Psi(8c u(s)) \dd{s}
			\]
			holds for every nonnegative measurable function $u \in \sobo_0^{1,\Psi}(\Omega;V)$.
		\end{enumerate}
		
	\end{lemma}
	We are now ready to give the:
	\begin{proof}[Proof of Theorem \ref{lem:cianchi_ver}]
		Even though the operator $\mathbb{A}$ is more general than the symmetric gradient $\varepsilon$ as considered in \cite{cianchi_korn_2014}, the lemma follows by similar means. However, to keep the paper self-contained, we will sketch the necessary modifications to account for the more general operators $\mathbb{A}$. 
		
		Let $u\in\hold_{c}^{\infty}(\R^{n};V)$. Since $\mathbb{A}$ is elliptic, we can express every  partial derivative $\partial_j u$, $j\in\{1,...,n\}$, in terms of $\mathbb{A} u$ by employing Fourier multipliers as follows:
		\[
		\partial_j u = c_{j}\mathscr{F}^{-1} \left[\xi_j(\mathbb{A}^*[\xi] \mathbb{A}[\xi])^{-1} \mathbb{A}^*[\xi] \widehat{\mathbb{A} u}\right] \overset{\text{Theorem } \ref{thm:avg}}{=:} L_j \mathbb{A}u +  K_j \ast \mathbb{A} u,
		\]
		where $c_{j}\in\mathbb{C}$, $K_{j}\colon \R^{n}\setminus\{0\}\to\mathrm{Lin}(W;V)$ represents a Calderón-Zygmund kernel, and $L_{j}\in\mathrm{Lin}(W;V)$ is a fixed linear map. This follows from Theorem \ref{thm:avg} and the fact that
		\[
		m_j := \xi_j(\mathbb{A}^*[\xi] \mathbb{A}[\xi])^{-1} \mathbb{A}^*[\xi],\qquad \xi\in\R^{n}\setminus\{0\}, 
		\]
		is a Fourier multiplier that belongs to the class $\hold^\infty(\R^n \setminus \{0\}; \text{Lin}(W;V))$ and is homogeneous of degree zero. According to Theorem \ref{thm:bagby}, we have the pointwise estimate 
		\[
		|\nabla u|^*(s) \leq C \Big( \frac{1}{s} \int \limits_0^s |\mathbb{A}u|^*(r) \, \dd{r} + \int \limits_s^{|\Omega|} |\mathbb{A} u|^*(r) \frac{\dd{r}}{r} \Big) + C|\mathbb{A} u|^*(s)
		\]
		for some constant $C>0$ and all $s>0$. Since $\Phi$ is monotonically increasing, this inequality implies that \begin{align*}
			E & := \int \limits_0^{|\Omega|} \Phi(|\nabla u|^*(s)) \dd{s}  \\ & \leq \frac{1}{2} C \int \limits_0^{|\Omega|} \Phi\Big(\frac{2}{s} \int \limits_0^s |\mathbb{A}u|^*(r)\dd{r} + 2\int \limits_s^{|\Omega|} |\mathbb{A}u|^*(r) \frac{\dd{r}}{r} \Big) \dd{s} + \frac{C}{2}\int \limits_0^{|\Omega|} \Phi(2|\mathbb{A}u|^*(s)) \dd{s} . 
		\end{align*}
		Since $\Psi$ dominates $\Phi$ globally, we have \[
		\int \limits_0^{|\Omega|} \Phi(2|\mathbb{A}u|^*(s)) \dd{s}  \leq \int \limits_0^{|\Omega|} \Psi(C|\mathbb{A}u|^*(s)) \dd{s}. 
		\]
		Now we conclude by Lemma \ref{lem:cianchi_hardy} and convexity of $\Phi$ that \begin{flalign*}
			\int \limits_0^{|\Omega|} \Phi  & \left( \frac{2}{s}  \int \limits_0^s |\mathbb{A}u|^*(r) \dd{r}  + 2\int \limits_s^{|\Omega|} |\mathbb{A}u|^*(r) \frac{\dd{r}}{r} \right) \dd{s} \\ & \leq \frac{1}{2} \left(  \int \limits_0^{|\Omega|} \Phi \left( \frac{4}{s} \int \limits_0^s |\mathbb{A}u|^*(r) \dd{r} \right) \dd{s} + \int \limits_0^{|\Omega|} \Phi \left( 4 \int \limits_s^{|\Omega|} |\mathbb{A}u|^*(r) \frac{\dd{r}}{r} \right) \dd{s} \right)
			\\ & \leq \frac{1}{2} \left( \int \limits_0^\infty \Psi(8c|\mathbb{A}u|^*(r)) \dd{s} +  \int \limits_0^\infty \Psi(32c|\mathbb{A}u|^*(r)) \dd{s} \right)
			\\ & \leq  \int \limits_0^\infty \Psi(32c|\mathbb{A}u|^*(r)) \dd{s}.
		\end{flalign*}
		Now it remains to state that the expressions are invariant under rearrangement, so that we have \[ \int \limits_\Omega \Phi(|\nabla u|) \, \dd{x} \lesssim \int \limits_\Omega \Psi(32c|\mathbb{A}u|) \, \dd{x}. \]
		This is what we wanted to show. 
	\end{proof}
	For future reference, we single out the following example:
	\begin{beisp} \label{ex:Korn}
		In order to emphasize why the above reduction method is optimally used with the function $t \mapsto t\log(1+t)$, we would like to take up \cite[Example 3.8]{cianchi_korn_2014} and consider it in our case. If $\mathbb{A}$ is elliptic, it follows from Lemma \ref{lem:cianchi_ver} that for $\alpha \geq 0$ \[
		\int \limits_\Omega |\nabla u| (\log(1 + |\nabla u|))^\alpha \, \dd{x} \leq C \int \limits_\Omega |\mathbb{A} u| \left( \log \left(1 + |\mathbb{A} u| \right) \right)^{\alpha+1} \, \dd{x},\qquad u\in\hold_{c}^{\infty}(\Omega;V),
		\]
		for a constant $C > 0$. This shows that a logarithm is lost due to the Korn inequality. In particular, $\alpha = 0$ appears as a limiting case. 
	\end{beisp}
	\subsection{The reduction argument for $L\log L$-growth}
	
	Again, we define $ G: W \to \R $ by $G\colon z \mapsto F(\pi_{\mathbb{A}}(z)) $. Then, as in \eqref{eq:juhu}, we have
	\begin{align*}
		\int \limits_{\R^n} G(z_0 + \nabla \zeta) - G(z_0) \, \dd{x} & = \int \limits_{\R^n} F(\pi_{\mathbb{A}} (z_0) + \mathbb{A} \zeta) - F(\pi_{\mathbb{A}} (z_0)) \, \dd{x} \\ &
		\!\!\!\!\!\!\!\overset{{\text{Thm.}\, \ref{lem:quasiconvex}}}{\geq}  \nu_M \int \limits_{\R^n} V_\varphi(\pi_{\mathbb{A}}(z_0) + \mathbb{A}\zeta ) - V_\varphi(\pi_{\mathbb{A} } (z_0) )\dd{x} \\
		& = \nu_M \int \limits_{\R^n} V_\varphi(\pi_{\mathbb{A}}(z_0) + \mathbb{A}\zeta ) - V_\varphi(\pi_{\mathbb{A} } (z_0) ) \\ & \hspace{20ex} - \langle (\nabla V_\varphi)(\pi_\mathbb{A}(z_0)), \mathbb{A}\zeta \rangle \dd{x},
	\end{align*}
	for all $\zeta \in \hold_c^\infty((0,1)^n; V)$ and $|z_0|\leq M$.
	The goal here is once again to estimate the final terms from below, thereby establishing that $ G $ is $ V_1 $-quasiconvex, but with a \emph{different} function $ V_1 $ than $ V_\varphi $, specifically $$ V_1(z) := \sqrt{1+|z|^2} - 1  .$$
	It is natural to consider this function $V$, since, as mentioned in the introduction, there is a loss of one logarithm due to the Korn-type inequality from Theorem \ref{lem:cianchi_ver}. To achieve this estimate, we will employ a suitable auxiliary function. Let  \[
	\Psi(t) := \begin{cases}
		\delta_1 t^2 ,& t \leq \max \{3M, 1\} \\
		\varphi(1+|\pi_\mathbb{A}(z_0)| +t) - \varphi(1+|\pi_\mathbb{A}(z_0)|) - \delta_2, & t>\max \{3M,1\}
	\end{cases}
	\]
	and \[
	\Phi(t) := \begin{cases}
		\delta_3 t^2 , & t \leq \max \{3M, 1\}  \\
		(1+|\pi_\mathbb{A}(z_0)|^2)^{-\frac{3}{2}} \left( \sqrt{1+t^2} - 1 \right) - \delta_4, & t > \max \{3M, 1\} ,
	\end{cases}
	\]
	where $\delta_1, \delta_3 >0$, $\delta_2,\delta_4 \in \mathbb{R}$,  are chosen such that $\Psi$ and $\Phi$ are Young functions. 
	By Lemma \ref{lem:aux1} and Lemma \ref{lem:aux2} we get that $$ \int \limits_{(0,1)^n} \Psi(|\mathbb{A\zeta}|) \dd{x} \simeq\int \limits_{(0,1)^n} V_\varphi(\pi_{\mathbb{A}}(z_0) + \mathbb{A}\zeta ) - V_\varphi(\pi_{\mathbb{A} } (z_0) ) \dd{x}.$$
	Our next aim is to derive an analogous relation for $\Phi$ and $V_1$. This requires the next lemma: 
	\begin{lemma}[{{\cite[Lemma 4.1]{gmeineder_partial_2019-1}}}] \label{lem:comp}
		Let $E(z) := V_1(z+z_0) - V_1(z_0) - \langle (\nabla V)(z_0), z \rangle$. Then we have
		\[
		E(z) \geq \frac{1}{4} (1+|z_0|^2)^{-\frac{3}{2}} V(z).
		\]
	\end{lemma}
	\begin{lemma} 
		It holds  $$ \int \limits_{(0,1)^n} \Phi(|\mathbb{A\zeta}|) \dd{x} \simeq\int \limits_{(0,1)^n} V_1(\pi_{\mathbb{A}}(z_0) + \mathbb{A}\zeta ) - V_1(\pi_{\mathbb{A} } (z_0) ) \dd{x}.$$
	\end{lemma}
	\begin{proof}
		Let $S := \max \{3M,1\}$. We split the integral once again by \begin{align*}
			\int \limits_{(0,1)^n} \Phi(|\mathbb{A} \zeta|) \, \mathrm{d}x 
			&= \int \limits_{(0,1)^n \cap \{|\mathbb{A} \zeta| \leq S\}} \Phi(|\mathbb{A} \zeta|) \, \mathrm{d}x + \int \limits_{(0,1)^n \cap \{|\mathbb{A} \zeta| \geq S\}} \Phi(|\mathbb{A} \zeta|) \, \mathrm{d}x \\
			&=: A + B.
		\end{align*}
		By Lemma \ref{lem:aux2} we conclude that \[
		A \simeq \int \limits_{(0,1)^n \cap \{|\mathbb{A} \zeta| \leq S\}} V_1(\pi_{\mathbb{A}}(z_0) + \mathbb{A}\zeta ) - V_1(\pi_{\mathbb{A} } (z_0) )  \, \mathrm{d}x.
		\]
		Since $\xi$ is compactly supported, we can use Lemma \ref{lem:comp} and get the estimate \begin{align*}
			B & = 4 \int \limits_{(0,1)^n \cap \{|\mathbb{A} \zeta| \geq S\}} \frac{1}{4} (1+|\pi_{\mathbb{A}}(z_0)|^2)^{-\frac{3}{2}} V_1(\mathbb{A} \zeta) \dd{x} \\
			& \overset{\text{Lemma \ref{lem:comp}}}{\leq} 4\int \limits_{(0,1)^n \cap \{|\mathbb{A} \zeta| \geq S\}} V_1(\pi_{\mathbb{A}}(z_0) + \mathbb{A}\zeta ) - V_1(\pi_{\mathbb{A} } (z_0) ) \dd{x}.
		\end{align*}
		On the other hand by Lipschitz continuity of $V_1$ \begin{align*}
			\int \limits_{(0,1)^n \cap \{|\mathbb{A} \zeta| \geq S\}} V_1(\pi_{\mathbb{A}}(z_0) + \mathbb{A}\zeta ) & - V_1(\pi_{\mathbb{A} } (z_0) ) \dd{x} \\ & \leq \int \limits_{(0,1)^n \cap \{|\mathbb{A} \zeta| \geq S\}}  |\pi_{\mathbb{A}}(z_0) + \nabla \xi - \pi_{\mathbb{A}}(z_0)| \dd{x} \\
			& = \int \limits_{(0,1)^n \cap \{|\mathbb{A} \zeta| \geq S\}}  |\mathbb{A} \xi | \dd{x} \\
			& \leq \int \limits_{(0,1)^n \cap \{|\mathbb{A} \zeta| \geq S\}}  5 \left( \sqrt{1+|\mathbb{A} \zeta|^2} -1 \right)  \dd{x} \\
			& = 5B.
		\end{align*}
		Therefore, the proof is complete.
	\end{proof}
	We conclude that there are constants $c_1, c_2, c_3, c_4, c_5 > 0$ only depending on $\delta,M$, such that \[
	\Psi(t) \simeq \begin{cases}
		t^2 & t \leq c_1 \\
		(t+c_3) \log(t+c_2)-c_4 & t > c_1
	\end{cases}
	\]
	and \[
	\Phi(t) \simeq \begin{cases}
		t^2 & t \leq c_1 \\
		t - c_5 & t > c_1
	\end{cases}
	\]
	for all $t \geq 0$. In particular, none of the constants depends on $z_0$.
	\begin{lemma} \label{lem:abschätz_phi_psi}
		There are $c,\tilde{c} > 0$ such that for all $t \geq 0$, it holds \[
		\Phi^*(t) \gtrsim \begin{cases}
			\max \{\frac{t^2}{4}, c_1(t-\tilde{c})+c_2 \} & t \leq \tilde{c} \\
			\infty & t > \tilde{c}
		\end{cases}
		\]
		and \[
		\Psi^*(t) \lesssim \max \left\{ \frac{t^2}{4}, c_4 + c\exp \left(\frac{t}{c}-1 \right) \right\}.
		\]
	\end{lemma}
	
	\begin{proof}
		First, we compute $\Phi^*(t)$:
		$$
		\Phi^*(t) = \sup_{s \geq 0} \left( st - \Phi(s) \right).
		$$
		Depending on the value of $s$, we have two cases:
		\begin{enumerate}
			\item  $s \leq c_1$:
			Since $
			\Phi(s) \simeq s^2$ we have $ st - \Phi(s) \geq st - cs^2,
			$
			for some $c>0$.
			To find the maximum, we set the derivative $\frac{\mathrm{d}}{\mathrm{d}s}(st - cs^2) = t - 2cs$ to zero:
			$$
			t - 2cs = 0 \implies s = \frac{t}{2c}.
			$$
			This gives us
			$$
			st - c s^2 = \frac{t}{2c} \cdot t - c\left(\frac{t}{2c}\right)^2 = \frac{t^2}{4c}.
			$$
			\item $s > c_1$: Since $\Phi(s)  \simeq s - c_5$, we can choose $\tilde{c}>0$ such that $$st - \Phi(s) \geq st - \tilde{c}(s - c_5) = st - \tilde{c}s + \tilde{c}c_5 = s(t-\tilde{c}) + \tilde{c}c_5.$$
			Here, the result depends on whether $t \leq \tilde{c}$ or $t > \tilde{c}$.
			\begin{enumerate}
				\item For $t \leq \tilde{c}$, the term $s(t-\tilde{c})$ is not positive, so the supremum is
				$$
				c_1(t-\tilde{c}) + c_5.
				$$
				\item If $t>\tilde{c}$, then $(t-\tilde{c})$ is positive, therefore $$s(t-\tilde{c})+c_5 \to \infty$$ 
				for $s \to \infty$.
			\end{enumerate}
		\end{enumerate}
		In summary, we get:
		$$
		\Phi^*(t) \gtrsim \begin{cases}
			\max \left\{ \frac{t^2}{4}, c_1(t-\tilde{c})+c_5 \right\} & t \leq \tilde{c} \\
			\infty & t > \tilde{c}.
		\end{cases}
		$$
		Now, we compute $\Psi^*(t)$:
		$$
		\Psi^*(t) = \sup_{s \geq 0} \left( st - \Psi(s) \right).
		$$
		Depending on the value of $s$, we have two cases:
		\begin{enumerate}
			\item $s \leq c_1$:
			In this case $
			\Psi(s) \simeq s^2$ implies that there is a $c>0$ such that, $st - \Psi(s) \leq st - cs^2.$
			As before, the maximum of the right-hand side is
			$\frac{t^2}{4c}
			$
			and is therefore an upper bound for the left-hand side.
			\item $s > c_1$:
			Since 
			$\Psi(s)  \simeq (s+c_3) \log(s+c_2) - c_4$, we find a $c>0$ such that $$st - \Psi(s) \leq st - c\left[(s+c_3) \log(s+c_2) - c_4\right].
			$$
			An upper bound for this can be obtained by using $\log(s+c_2) \geq \log(s)$ and $(s+c_3) \geq s$:
			$$
			st - c(s+c_3) \log(s+c_2) + cc_4 \leq st - cs \log(s) + cc_4.
			$$
			To find the maximum of $st - cs \log(s)$, we set the derivative to zero:
			\begin{align*}
				\frac{\mathrm{d}}{\mathrm{d}s} (st - cs \log(s)) & = t - c\log(s) - c = 0 \\ & \implies \log(s) = \frac{t}{c} - 1 \\ & \implies s = \exp \left(\frac{t}{c} - 1 \right).
			\end{align*}
			Since the second derivative $\frac{\mathrm{d}^2}{\mathrm{d}s^2} (st - c s \log(s)) = -\frac{c}{s}$ is negative for $s > 0$, this is a maximum. 
			Evaluating at $s = \exp \left(\frac{s}{c} - 1 \right)$, we get:
			\begin{align*}
				st - c s \log(s) & = t \exp \left(\frac{t}{c}-1 \right) - c\exp \left(\frac{t}{c}-1 \right) \log \left(\exp\left(\frac{t}{c}-1 \right) \right) \\ & = t \exp \left(\frac{t}{c}-1\right) - c\exp \left(\frac{t}{c}-1 \right) \left(\frac{t}{c}-1\right) \\ & = c\exp \left(\frac{t}{c}-1 \right).
			\end{align*}
			Thus, we get:
			$$
			\Psi^*(t) \lesssim\max \left\{ \frac{t^2}{4}, c_4 + c\exp \left(\frac{t}{c}-1 \right) \right\}.
			$$
			This completes the proof.
		\end{enumerate}
		
	\end{proof}
	
	\begin{lemma}\label{lemma:cond}
		There is a (sufficiently large) $c > 0$, such that 
		\begin{equation}
			t \int \limits_{0}^t \frac{\Psi^*(s)}{s^2} \dd{s} \lesssim \Phi^*(ct)
		\end{equation}
		and \begin{equation}
			t \int \limits_{0}^t \frac{\Phi(s)}{s^2} \dd{s} \lesssim \Psi(ct)
		\end{equation}
		holds for all $t \geq 0$.
	\end{lemma}
	\begin{proof}
		\begin{enumerate}
			\item Based on Lemma \ref{lem:abschätz_phi_psi}, it is natural to consider the cases $t\leq \tilde{c}$ and $t>\tilde{c}$. \begin{enumerate}
				\item For $t \leq \tilde{c}$, we have \begin{align*}
					t \int \limits_{0}^t \frac{\Psi^*(s)}{s^2} \dd{s}  &\lesssim t \int \limits_0^t \frac{s^2}{4s^2} \dd{s}  = \frac{t^2}{4} \\ & \leq \max \left\{ \frac{t^2}{4}, c_1(t-\tilde{c})+c_2 \right\} \lesssim \Phi^*(t).
				\end{align*}
				\item For $t>\tilde{c}$, the right-hand side is infinite, so there is nothing to show.
			\end{enumerate}
			\item Again, we distinguish between $t \leq c_1$ and $t>c_1$.
			\begin{enumerate}
				\item If $t \leq c_1$, then we have \begin{align*}
					t \int \limits_{0}^t \frac{\Phi(s)}{s^2} \dd{s}  \simeq t \int \limits_{0}^t \frac{s^2}{s^2} \dd{s} = t^2 \simeq \Psi(t).
				\end{align*}
				\item If $t > c_1$, then \begin{align*}
					t \int \limits_{0}^t \frac{\Phi(s)}{s^2} \dd{s}  & \simeq t \left( \int \limits_{0}^1 \frac{s^2}{s^2} \dd{s} + \int \limits_{1}^t \frac{s-c_5}{s^2} \dd{s}  \right) \dd{s} \\ & \leq t + t\log(t) + c_5- tc_5.
				\end{align*}
				Now choose \[
				c > \max  \left\{ \frac{\exp(\frac{c_5(1-c_1)}{c_3}) - c_2}{c_1}, \frac{e-c_2}{c_1}, 2 \right\}.
				\]
				Then \[
				t + \log(t) \leq ct\log(ct+c_2)
				\]
				and \[
				c_3 \log(ct+c_2) - c_4 \geq c_5 - t c_5.
				\]
				With this, we get \[
				t \int \limits_{0}^t \frac{\Phi(s)}{s^2} \dd{s} \lesssim \Psi(ct)
				\]
				as claimed.
			\end{enumerate}
		\end{enumerate}    
	\end{proof}

	\subsection{Proof of Theorem \ref{OrliczMainThm2}}

	We are now ready to give the: 
	\begin{proof}[Proof of Theorem \ref{OrliczMainThm2}]
		{As noted in \cite{breit_trace-free_2017}, for a function $\varphi \in \Delta_2$, we have
			\begin{align*}
				\sobo_{0}^{1,1}(\Omega;V)\cap\sobo^{1,\varphi}(\Omega)=\sobo_{0}^{1,\varphi}(\Omega), 
			\end{align*}
			if $\Omega$ is a bounded Lipschitz domain.
			The situation for elliptic operators is similar, but the corresponding Korn inequality entails the loss of a logarithm. Nevertheless, combining Example \ref{ex:Korn} with extending the function to the entire space and subsequently localizing yields
			\begin{align}\label{eq:traceid}
				\sobo_{0}^{1,1}(\Omega;V)\cap\sobo^{\mathbb{A},\varphi}(\Omega)=\sobo_{0}^{\mathbb{A},\varphi}(\Omega), 
			\end{align} if $\mathbb{A}$ is an elliptic operator.
		}
		Using Lemma \ref{lem:cianchi_ver}, we get that \[
		\int \limits_{(0,1)^n} \Phi (|\nabla \zeta|)  \dd{x} \simeq \int \limits_{(0,1)^n} \Phi \left(\frac{1}{C}|\nabla \zeta| \right) \dd{x} \lesssim \int \limits_{(0,1)^n} \Psi(|\mathbb{A}\zeta|) \dd{x}. 
		\]
		We again use the reduction strategy from the proof of Theorem \ref{OrliczMainThm}. Since 
		\[
		\int_{(0,1)^n} \big( V_1 \left(z_0 + \nabla \zeta \right) - V_1 \left(z_0\right) \big) \, \mathrm{d}x \simeq 
		\int_{(0,1)^n} \Phi (|\nabla \zeta|) \, \mathrm{d}x  
		\]
		and $\pi_\mathbb{A}$ is bounded, it follows that 
		\[
		|G(z)| \leq c(1 + |z|^q)
		\]
		for $G$ as defined in \eqref{eq:G} and all $q > 1$. {Let $\mathcal{G}$ be the functional defined as in \eqref{fct:G} and fix some $1 < q < \frac{n}{n-1}$.} Applying \cite[Theorem 2.1]{gmeineder_quasiconvex_2024}, we deduce that every local $\mathrm{BV}$ minimizer with compactly supported variation of the weak*-relaxed functional 
		$\bar{\mathcal{G}}^*[w; \Omega]$ (see \cite[{Definition 1.2}]{gmeineder_quasiconvex_2024})
		is $\hold^{1,\alpha}$-partially regular. Therefore, we need to show that every minimizer $u \in \sobo_\text{loc}^{\mathbb{A},\varphi}(\Omega)$ of $\mathcal{G}$ is a local $\mathrm{BV}$-minimizer of $\bar{\mathcal{G}}_{{u}}^*$. { Recall that for $\Omega \Subset \Omega'$ \begin{multline*}
				\bar{\mathcal{G}}_{u_0}^*[\bar{u}; \Omega; \Omega'] := \inf \Big\{ \liminf_{j \to \infty} \int_{\Omega'} G(\nabla u_j) \dd{x}: 
				(u_j) \in \sobo^{1,q}(\Omega'; V), \\
				\bar{u}=u_0 \text{ in } \Omega' \backslash \bar{\Omega}, \quad u_j \rightharpoonup^* \bar{u} \text{ in } \mathrm{BV}(\Omega'; V) \Big\},
			\end{multline*}
			see \cite[Equation 1.9]{gmeineder_quasiconvex_2024}, and for $v \in \sobo^{1,q}(\Omega;V)$  \[
			\bar{\mathcal{G}}_{v}^*[u; \Omega] := \bar{\mathcal{G}}_{u_0}^*[\bar{u}; \Omega; \Omega'] - \mathcal{G}[u; \Omega' \backslash \Omega],
			\]
			where $u_0 \in \sobo^{1,q}_0(\Omega';V)$ is an extension of $v$ to $\Omega'$ and $\bar{u}$ is an extension of $u$ to $\Omega'$ by $u_0$, see \cite[Equation 1.10]{gmeineder_quasiconvex_2024}.
		} 
		
		Since $u \in\sobo_\text{loc}^{\mathbb{A},\varphi}(\Omega)$ we have $u \in\sobo_\text{loc}^{1,1}(\Omega{; V})$ and therefore, {by \cite[Corollary 4.2]{gmeineder_quasiconvex_2024}}, for every $x_0 \in \Omega$ we find an $r>0$ such that  
		\[
		\tr_{\partial B_r(x_0)}(u) \in \sobo^{1-\frac{1}{q},q}(\partial B_r(x_0);V).
		\]
		Therefore, we can choose $u_0 \in \sobo^{1,q}(\Omega; V)$, such that $u-u_0 \in \sobo^{\mathbb{A},\varphi}_0(B_r(x_0); V)$ {by \eqref{eq:traceid}}.
		As a result, there exists a sequence $(\psi_j) \in \hold_c^\infty(\Omega;V)$ with $\operatorname{supp}(\psi_j) \subset B_r(x_0)$ for all $j \in \mathbb{N}$ and \[
		\| (u-u_0) - \psi_j \|_{\sobo^{\mathbb{A},\varphi}(B_r(x_0))} \to 0. 
		\]
		{Applying Example~\ref{ex:Korn}} and considering a non-relabeled
		subsequence, we also observe that $u_0 + \psi_j \rightharpoonup^* \bar{u}$ in $\mathrm{BV}(\Omega;V)$. 
		
		By setting $\vartheta_j := 1+ |\mathbb{A}(u_0 + \psi_j)| + |\mathbb{A}u|$ and $\theta_j := |\mathbb{A}u - (\mathbb{A}(u_0 + \psi_j))|$, we can argue as in the proof of \cite[Theorem 10.7]{gmeineder_quasiconvex_2024} and get \[
		\bar{\mathcal{G}}_u^*[u; B_r(x_0)] \leq \liminf \limits_{j \to \infty} \int \limits_{B_r(x_0)} F(\mathbb{A}(u_0 + \psi_j)) \dd{x} \leq  \int \limits_{B_r(x_0)} F(\mathbb{A}u) \dd{x}.
		\]
		{Proceeding along the lines of \cite[Theorem 10.7]{gmeineder_quasiconvex_2024},} let $\varphi \in \mathrm{BV}_c(B_r(x_0);V)$ and let $(v_j) \subset u_0 + \sobo_0^{1,q}(B_r(x_0);V)$ be a generating sequence for 
		$\bar{\mathcal{G}}_{u_0}^*[u+\varphi; B_r(x_0)]$. Then $v_j - u \in \sobo_0^{\mathbb{A},\varphi}(\Omega)$ by {\eqref{eq:traceid}}. By local minimality we get \[
		\int \limits_{B_r(x_0)} F(\mathbb{A}u) \dd{x} \leq \int \limits_{B_r(x_0)} F(\mathbb{A}u + \mathbb{A}(v_j - u)) \dd{x} = \int \limits_{B_r(x_0)} F(\mathbb{A}v_j) \dd{x}.
		\]
		Passing to the limit yields \[
		\int \limits_{B_r(x_0)} F(\mathbb{A}u) \dd{x} \leq \limsup \limits_{j \to \infty} \int \limits_{B_r(x_0)} F(\mathbb{A}v_j) \dd{x} = \bar{\mathcal{G}}_{u}^*[u+\varphi; B_r(x_0)].
		\]
		We conclude that $u$ is a local $\mathrm{BV}$-minimizer for compactly supported variations.
		This completes the proof of Theorem \ref{OrliczMainThm2}.
	\end{proof}
	
	\section*{Acknowledgment}
	{The author wants to thank Franz Gmeineder for the project idea, fruitful discussions and help with the preparation of the manuscript. Moreover, financial support through a stipend of the Landesgraduiertenf\"{o}rderungsgesetz Baden-W\"{u}rttemberg is gratefully acknowledged. Lastly, the author is grateful to the anonymous referee for valuable suggestions which led to an overall improvement of the paper. } 
	\bibliographystyle{abbrv}
	\bibliography{main.bib}
	
\end{document}